\documentclass[12pt,twoside, english]{article}
\usepackage{psfrag,babel}
\usepackage{latexsym}
\usepackage{amsmath,amsthm}
\usepackage{amsfonts}
\usepackage{amssymb}
\usepackage{graphicx}
\usepackage{color}
\usepackage{epsfig}

\usepackage[colorlinks=true]{hyperref}

\hypersetup{urlcolor=blue, citecolor=red}

\newtheorem{theorem}{Theorem}[section]
\newtheorem{corollary}[theorem]{Corollary}
\newtheorem{lemma}[theorem]{Lemma}
\newtheorem{proposition}[theorem]{Proposition}

\newtheorem{definition}[theorem]{Definition}

\newtheorem{maintheorem}{Theorem}

\DeclareMathOperator{\dist}{dist}

\begin{document}

%%%%%%%%%%%%%%%%%%%%%%%%%%%%%%%%%%%%%%%%%%%%%%%%%%%%%%%%
% Title and author(s)'s name(s)
%%%%%%%%%%%%%%%%%%%%%%%%%%%%%%%%%%%%%%%%%%%%%%%%%%%%%%%%
\title{Entropy formula for $C^1$ expanding maps}
\author
 {Eleonora Catsigeras\thanks{Instituto de Matem\'{a}tica y Estad\'{\i}stica \lq\lq Prof. Rafael Laguardia", Universidad de la Rep\'{u}blica, Uruguay.  Postal Address: Av. Herrera y Reissig 565, C.P.11300 Montevideo, Uruguay. E-mail: eleonora@fing.edu.uy}  \, and Fernando Valenzuela G\'{o}mez\thanks{E-mail: fagomez7@hotmail.com}
 }

\date{\small April 16, 2026}
\maketitle

\begin{abstract} 
We prove that the (necessarily existing) pseudo-physical or SRB-like measures of $C^1$  expanding dynamical systems on a compact Riemannian manifold satisfy Pesin's entropy formula.  We include examples of $C^1\setminus C^{1+\alpha}$ expanding maps on the circle and on the 2 torus and study their pseudo-physical measures.  
\end{abstract}



\section{Introduction}

We consider the dynamical system by iteration of a map $f:M \to M$ of $C^1$ class on a compact Riemannian manifold $M$ of finite dimension. The search for \lq\lq natural\rq\rq \ invariant measures that describe the statistical behavior of an observable set of orbits of the system is a key point in the ergodic theory. Usually, the concept of observability of a set of orbits is associated to its volume, namely, a set of orbits is observable if it has positive Lebesgue measure. This is particularly restrictive if the map is not Lebesgue preserving.  The ergodicity of an invariant measure, if the system is not volume preserving, does not ensure that the measure is \lq\lq natural\rq\rq. In fact, an ergodic measure $\mu$ describes the statistical behavior of $\mu$-almost all the orbits, but if $\mu$ is mutually singular with the Lebesgue measure, the set of such orbits may zero volume.

\subsection{Physical and pseudo-physical measures.}

We denote by $\mathcal M$ the set of Borel probability measures on $M$, endowed with the weak$^*$- topology (see for instance \cite{Viana-Oliveira-Book}).  The weak$^*$ topology in $\mathcal M$ is defined by the following equality:
$$ {\lim_{n \rightarrow + \infty}}^* \  \mu_n = \mu   \mbox{ for } \mu_n, \mu \in {\mathcal M} \ \mbox{ if and only if } $$
$$   \lim_{n \rightarrow + \infty}\int \phi \, d \mu_n =\int \phi \, d \mu   \mbox{ for all the continuous functions } \phi: M \to \mathbb{R}. $$

We denote by ${\mathcal M}_f \subset {\mathcal M}$ the subset of  measures $\mu $ that are invariant by $f$, i.e. $f^* \mu= \mu$, where $f^*:{\mathcal M} \to {\mathcal M}$ is the pull-back of $f$, defined by $f^*\mu(B)= \mu(f^{-1} (B)) $ for any Borel-measurable set $B$.

In \cite{misiurewiczNaturalMeasures} a measure $\mu \in \mathcal M_f$ is called \em natural \em if it satisfies
$$\mu = {\lim_{n\rightarrow +\infty}}^* \frac{1}{n} \sum_{j=0}^{n-1} {(f^*)}^j \nu ,  $$
for some (not necessarily invariant) Borel probability measure $\nu$ that is absolutely continuous with respect to the Lebesgue measure. The problem in this definition of natural measures, is that they do not necessarily exist.

One of the most used concept of relevance of an invariant measure, from the statistical viewpoint for a positive volume set of orbits,   is the property of being \lq\lq physical\rq\rq,  that we will define below:

\begin{definition} {\bf (Empiric probabilities and basin of statistical attraction.)}
	\label{DefinitionStatisticalBasin} 
	
	\em
	For any initial point $x \in M$ the \em empiric probability measure \em $\sigma_n(x) \in {\mathcal M}$ up to time $n$ of the orbit of $x$ is 
	$$ \sigma_n(x) := \frac{1}{n} \sum_{j=0}^{n-1} \delta_{f^j(x)},   $$
	where $\delta_y$ denotes the Dirac delta probability measure supported on $y$.  In other words, the empiric probability is equally supported on the points of the finite piece of orbit from $x$ up to  $f^{n-1}(x)$.
	
	For any point $x \in M$ the sequence  $\{\sigma_n(x)\}_{n\geq1} \subset {\mathcal M}$ of empiric probability has convergent subsequences, because ${\mathcal M}$ is a compact space with the weak$^*$-topology.
	We call the set of probabilities measures that are the limits$^*$ of the convergent subsequences of $\{\sigma_n(x)\}_{n\geq1} $, the \em $p$-omega limit \em of the orbit of $x$ (i.e. the omega limit in the space of probabilities), and denote it by   $p\omega(x) $. Precisely,
 	$$ p\omega(x) = \{\nu \in \mathcal M: \nu = {\lim}^*_{j \rightarrow +\infty} \sigma_{n_j}(x) \ $$ $$ \mbox {for some convergent subsequence }  \{\sigma_{n_j}(x)\}_{j\geq 1} \}.  $$
	
	For any invariant measure $\mu$, the  \em  basin $B(\mu)$ of statistical attraction \em of $\mu$ is the set
	$$B(\mu) : = \{x \in M:  \mu = {\lim}^*_{n \rightarrow +\infty} \sigma_{n}(x) \} = \{ x \in M: p\omega(x) = \{\mu\}\}.$$
\end{definition}

\begin{definition} {\bf (Physical measures.)}
	\label{DefinitionPhysicalMeasure} 	\em 
	We call an invariant probability measure $\mu$ \em physical \em if its basin of statistical attraction has positive Lebesgue measure, i.e. 
	$$m(B({\mu})) > 0,$$
	where $m$ denotes the Lebesgue measure.

\end{definition}

The physical measures are also called Sinai-Ruelle-Bowen (SRB) measures, due to the early works in the decade of 1970's of Ya. Sinai \cite{Sinai-1972}, D. Ruelle  and R. Bowen \cite{Bowen-Ruella-1975}, \cite{Bowen1975LectNotes}, \cite{Ruelle1976}, introducing the physical measures for smooth dynamical systems with uniform hyperbolicity. 

When working in the $C^1$ topology, we prefer to call them physical measures, instead of SRB-measures, to avoid confusions:  In fact, there is abundant literature studying the physical or SRB-measure for smooth systems or even $C^{1 + \alpha}$ systems with $\alpha >0$ (see Definition \ref{DefinitionC1+alpha}). With such regularity, relevant properties appear:  the conditional measures of the physical probabilities along the unstable submanifolds are absolutely continuous with respect to the Lebesgue measures of these submanifolds \cite{pesinTeoria} \cite{ledrappieryoung}. In the literature, these properties are required or proved, before calling the invariant measures SRB \cite{Young-WhatAreSRB}. But in our context, where regularity is only  $C^1$ (but not necessarily $C^{1 + \alpha}$), the existence of unstable submanifolds fails \cite{Pugh-1984}. Besides, also in the particular cases for which the unstable manifolds exist, the properties of absolute continuity do not hold \cite{avilabochi}.  

One of the most relevant problems in the ergodic theory, is to prove the existence of physical measures, since a priori the sequence of empiric probabilities may be non convergent for a set of orbits with positive volume.  The existence of physical measures, is mainly obtained in a scenario of some kind of uniform or non-uniform hyperbolicity or, at least, domination of the expanding directions. The existence of physical measures was proved for $C^1$-generic expanding map of the circle in \cite{Campbell-Quas}, for $C^1$ generic diffeomorphisms having an hyperbolic attractor in \cite{qiuhao}, and for $C^{1+\alpha}$ diffeomorphisms with dominated splitting in \cite{araujotahzibi2008}.  More recently, a characterization in the $C^{1 + \alpha}$ scenario of the existence of SRB measures on surfaces was given in \cite{Climenhaga-Luzzatto-Pesin-2022}.

Besides the existence, the problem of uniqueness or, at least finitude, of the physical measures is the object of research mainly for partially hyperbolic systems.  In \cite{Obata-2023}  it is proved the existence and uniqueness of SRB measures for certain class of $C^2$ partially hyperbolic systems.     In  \cite{Alves-Dias-Luzzatto-Pinheiro-2017} the authors prove the existence of at most a finite number of SRB-measures for a class of $C^{1+\alpha}$ partially
hyperbolic dynamical systems. 

As said above, the existence of physical or SRB measures was mainly proved for systems that are $C^{1 + \alpha}$ regular (and with some kind of hyperbolicity or expanding properties), except some few articles that explore their existence in the $C^1$ topology.  In an intermediate situation, in \cite{Boukhecham-2024} the author finds SRB measures for hyperbolic systems that are more regular than $C^1$ but with weaker regularity than $C^{1+ \alpha}$.

To overcome the problem of nonexistence of physical measures, a generalization of such measures,  was introduced in 
\cite{CatEnr-polaca}: the concept of pseudo-physical  or SRB-like measure, which we define in the following paragraphs. 

Recall Definition \ref{DefinitionStatisticalBasin} of the $p\omega$-limit set of an orbit in the space of probabilities  and of basin of statistical attraction of an invariant measure.

\begin{definition} {\bf (Epsilon-weak basin of statistical attraction.) }
	\label{definitionEpsilonWeakBasin} \em
	Choose a distance $\dist^*$ in $\mathcal M$ that endows the weak$^*$ topology.
	
	 For any $f$-invariant probability measure $\mu$, and for $\epsilon >0$, we call the following set  $B_{\epsilon} (\mu) \subset M$ the \em 
 $\epsilon$-weak basin of statistical attraction of $\mu$: 
	$$ B_{\epsilon} (\mu) := \{ x \in M:  \dist^*(p\omega(x), \mu) < \epsilon   \}.  $$

\end{definition}

We note that the basin of statistical attraction defined in \ref{DefinitionStatisticalBasin}, may not coincide with the zero-weak basin of statistical attraction. In fact, the weak$^*$-distance between $p\omega(x)$ and $\mu$ may be zero, but the sequence of empiric probabilities may not converge, and have convergent subsequences whose limits are different from $\mu$. 

\begin{definition} {\bf (Pseudo-physical measures.)}
	\label{DefinitionPseudoPhysical}
	\em 
	
	We call an invariant probability measure $\mu$  \em pseudo-physical or SRB-like, \em if its $\epsilon$-weak basin of statistical attraction has positive Lebesgue measure for all $\epsilon >0$. In brief
	$$ m(B_{\epsilon} (\mu)) >0 \ \ \mbox {for all }  \epsilon >0.  $$
\end{definition}
The  following  properties were proved in \cite{CatEnr-polaca}:

The pseudo-physical measures do always exist for any continuous $f$, and do not depend of the chosen $\dist^*$ in the space of probability measures (provided it endows the weak$^*$- topology). Besides for Lebesgue-almost all the orbits, any convergent subsequence of empiric probabilities converges to a pseudo-physical measure. In other words, the set of all the pseudo-physical measures completely describes the observable statistical behavior of the system if the criteria of observability is that of the orbits with positive volume.  Any physical measure is pseudo-physical, so pseudo-physicality is a generalization of physicality. If the set of all the pseudo-physical measures is finite, then all the pseudo-physical measures are physical. 

In Section \ref{Section5-Examples} we present examples of $C^1$-expanding maps (see Definition \ref{DefinitionExpandingMaps}) on the circle and on the torus, that are not $ C^{1+\alpha}$ for any $\alpha >0$, and study their pseudo-physical measures. 

\subsection{Equilibrium states and Pesin's entropy formula. }

Other important definition in  the study of  statistical properties of a dynamical system, coming from the statistical mechanics, is the concept of equilibrium states of a variational principle (see for instance \cite{Bowen1975LectNotes}, \cite{kellerBook}), and in particular the set of measures that satisfy Pesin's entropy formula.        We will state the definitions of those concepts in the following paragraphs.

To define  the equilibrium states we will use the \em metric entropy \em of the map $f $ with respect to an $f$-invariant probability measure $\mu$. For a definition of the metric entropy, see Section \ref{SectionEntropy} of this article. For a more detailed exposition of the properties of the entropy see also for instance the book  \cite{Viana-Oliveira-Book}.

\begin{definition} {\bf (Equilibrium States.) }
	\label{DefinitionEquilibriumStates} \em 
	
	Let $\psi: M \mapsto \mathbb{R}$ be a continuous function called \em potential. \em
	
	We call the following supremum $P(f, \psi)$ the \em pressure \em of $f$ with respect to the potential $\psi$:
	$$ P(f, \psi) := \sup_{\nu \in \mathcal{M}_f} \left\{ h_{\nu}(f) + \int \psi \, d \nu\right\},$$
	where  $h_{\nu}(f)  $ is the metric entropy of $f$ with respect to the $f$-invariant measure $\nu$.
	
	An $f$-invariant measure $\mu$ is \em an equilibrium state \em of $f$ with respect to the potential $\psi$ if
	$$  h_{\mu}(f) + \int \psi \, d \mu = P(f, \psi).$$
		
\end{definition}
In the decade of 1970', Sinai, Ruelle and Bowen proved important relations between the equilibrium states and the physical measures for smooth hyperbolic systems \cite{Sinai-1972}, \cite{Bowen-Ruella-1975}, \cite{Bowen1975LectNotes}, \cite{Ruelle1976}.  
More recently, \cite{Alvarez-Cantarino} proves the existence  of equilibrium states for certain type of partially hyperbolic endomorphisms of $C^1$ class, and for certain type of potentials. In \cite{Hemenway-2024}, the existence and uniqueness of  equilibrium states  are proved  for non-uniformly expanding skew products and for Hölder continuous potentials.  Also the uniqueness of equilibrium states, besides their existence, is proved in \cite{Pacifico-Yang-Yang-2022} for a class of flows satisfying a version of the specification property among other conditions. In \cite{Alves-Oliveira-Santana-2024} the existence of finitely meany ergodic equilibrium states for a type of non-uniformly expanding maps, with respect to Hölder continuous potentials.

The equilibrium states are mainly applied when the potential is related with the positive Liapunov exponents which translate to the tangent space the chaotic behavior of the dynamics by iterations of $f$. 
 
Oseledet's Theorem (see for instance \cite{Barreira-Pesin-Book}) states that for any $f$-invariant measure $\mu$, at almost all the points $x$ with respect to $\mu$ there exists a  splitting of the tangent space
$$ T_x M = \oplus_{i=1}^{k(x)} E_x^i  $$
into $Df$-invariant measurable subspaces $E_x^i$ along which the Liapunov exponents exist according with the following definition:

\begin{definition} {\bf (Liapunov exponents.)}
	\label{DefinitionExpLiapunov} \em 
The  \em  Liapunov exponent $\chi_i(x)$ \em of $f$ at the point $x \in M$ along the measurable $Df$-invariant  tangent subspace $E_x^i$ is:
$$\chi_i(x):= \lim_{n \rightarrow \pm \infty} \frac{\|Df^n_x v\|}{n} \ \ {\mbox{ for all }}  0 \neq v \in E_x^i.    $$
The Liapunov exponents are the exponential rate of increasing (if positive) or decreasing (if negative) of the vectors of the tangent space, when iterating $Df: TM \mapsto TM$.
\end{definition}

We denote by $$\sum_{i=1}^{m(x)} \chi_i^+(x)$$ the sum of the Liapunov exponents that are strictly larger than zero at the point $x$, counting each one as many times as its multiplicity. If all the Liapunov exponents are smaller or equal than zero, that sum is null.
The following Theorem is due to Margulis \cite{Margulis} and Ruelle \cite{ruelleIneq}, and states an upper bound for the metric entropy, related with the positive Liapunov exponents:

\begin{theorem} {\bf (Margulis-Ruelle inequality)}
	\label{TheoremRuelleIneq} 
	
	For any $f$-invariant measure $\mu$,
	$$   h_{\mu}(f) \leq  \int \sum_{i=1}^{m(x)} \chi_i^+(x) \, d \mu$$
\end{theorem}

For a proof, see for instance \cite{Viana-Oliveira-Book}.

 \begin{definition} {\bf (Pesin's  entropy formula)}
 	\label{DefinitionPesinFormula} \em 
 	An $f$-invariant measure $\mu$ satisfies  \em Pesin's entropy formula \em if
 		$$   h_{\mu}(f) =  \int \sum_{i=1}^{m(x)} \chi_i^+(x) \, d \mu$$
 \end{definition}
Measures satisfying Pesin's entropy formula may not exist. But if someone exists, its metric entropy is the maximum possible with respect to the chaotic behavior of $f$ that is expressed by the positive Liapunov exponents.

When the system has a continuous $Df$-invariant unstable sub-bundle $U \subset TM$, the integral of the sum of the positive Liapunov exponents equal the integral of 
$$\phi := \log |\det Df|_U|.$$  If this latter function is continuous, its opposite $\psi = -\phi$  can be used as  the potential to study the equilibrium states. In this case the pressure $P(f, \psi) \leq 0$, due to Margulis-Ruelle inequality. So the measures satisfying Pesin's entropy formula, if someone exists, are the equilibrium states of $f$ with respect to the potential $\psi$, and the pressure is zero.

Ya. B. Pesin \cite{pesinTeoria} early initiated the so called Pesin's Theory, proving important relations between the Liapunov exponents and the existence of measures satisfying Pesin's entropy formula for some smooth systems, is a scenario for which there exists invariant measures that have properties of absolute continuity with respect to the Lebesgue measure along the unstable submanifolds.

Later, Ledrappier and Young \cite{ledrappieryoung} proved that the condition of absolute continuity  used in Pesin's Theory is indeed a characterization of the measures (if they exist)
that satisfy Pesin's entropy formula, provided the system is of  $C^{1+\alpha} \, (\alpha >0)$ class. This characterization is relevant: it is the key point  in the later research proving the existence of measures satisfying Pesin's entropy formula.
For instance  in \cite{araujotahzibi2008} the existence of a SRB measure that satisfies Pesin's entropy formula is proved for $C^{1+\alpha}$ diffeomorphisms with dominated splitting. In \cite{Alves-Mesquita-2023} it is proved the existence and uniqueness of SRB measure satisfying Pesin'  entropy formula for Gibbs-Markov induced maps, that translate to a piecewise    $C^{1+ \alpha}$ dynamics.

The characterization of Ledrappier and Young of measures satisfying Pesin's entropy formula via the properties of absolute continuity with respect to the Lebesgue measure along the unstable submanifolds, do not hold for  $C ^1$ systems if they are not $C^{1+ \alpha}$. In fact, generic $C^1$ systems do not have measures with that property of absolute continuity \cite{Pugh-1984},   \cite{avilabochi}.  Nevertheless, under some kind  of hyperbolicity, $C^1$ systems still have measures satisfying Pesin's entropy formula: In \cite{tahzibiC1GenericPesinFormula}, A. Tahzibi proved that generic $C^1$ systems of dimension two have an invariant measure satisfying Pesin's entropy formula. Later, in \cite{suntianDominatedSplittingPesinFormula}, Sun and Tian extended  Tahzibi's result to $C^1$- generic volume-preserving diffeomorphisms in any dimension with a dominated splitting. In \cite{CatCerEnr-DomSplit} it is proved that the necessarily existing pseudo-physical measures satisfy Pesin's entropy formula for all the $C^1$ systems with dominated splitting in any dimension. In \cite{CatEnr-Portugaliae} it was proved the same result but for $C^1$ expanding maps in dimension one. And in \cite{Araujo-Santos-2019} it is proved the result for $C^1$ nonuniform expanding maps in any dimension.

\subsection{Statement of the result for expanding maps.}

\begin{definition} {\bf (Expanding maps.)} 
	\label{DefinitionExpandingMaps} \em 
	
	The $C^1$ map $f: M \to M$ is (uniformly) \em expanding \em if there exists a constant $\lambda >1$ such that
$$ \|Df_x(v)\| \geq \lambda \|v\| \ \ \forall \ (x,v) \in TM.$$	
\end{definition}

Recall Definition \ref{DefinitionExpLiapunov} of the Liapunov exponents. Since for an expanding map, the norm of all the vectors in the tangent space grow more than $\lambda >1$ at each iterate, the exponential rate of growing for the vectors of any direction, is larger than $\log \lambda >0$. So all the Liapunov exponents are positive.

\begin{theorem} {\bf (Liouville formula)}
\label{Theorem LiouvilleForm}	
For any $C^1$ map $f$ and for any $f$-invariant measure $\mu$ 
$$\int \sum_{i=1}^{k(x)} \chi_i (x) \, d \mu = \int \log |\det (Df)| \, d \mu,    $$ 
where  $ \sum_{i=1}^{k(x)} \chi_i (x)  $ is the sum of all the Liapunov exponents at the point $x$, counting each one as many times as its multiplicity. 
\end{theorem}
For a proof of Liouville formula, see for instance \cite{Teschl-Book}.

\begin{corollary}
	\label{CorollaryLiouvilleFormulaForExpandingMaps}
	If the $C^1$ map $f$ is expanding then, for any $f$-invariant measure $\mu$
$$	\int \sum_{i=1}^{m(x)} \chi^+_i (x) \, d \mu = \int \log |\det (Df)| \, d \mu, $$
where $ \sum_{i=1}^{m(x)} \chi^+_i (x)$ is the sum of all the positive Liapounov exponents at the point $x$, counting each one as many times as its multiplicity.

\end{corollary}
\begin{proof}
	Since the map is expanding, all the Liapunov exponents are positive.  Therefore, this corollary is  a restatement of Liouville formula. 
\end{proof}

\begin{proposition} \label{PropositionPesinForm=EqState}
	
	For a $C^1$ expanding map $f$, an invariant probability measure $\mu$ satisfies  Pesin's entropy formula if and only if it is an equilibrium state for the potential $\psi= -\log |\det (Df)|$ and the pressure $P(f, \psi) =0$.
\end{proposition}

\begin{proof}

	Due to  Theorem \ref{TheoremRuelleIneq} and  Corollary \ref{CorollaryLiouvilleFormulaForExpandingMaps}, we have
	$$ P(f, \psi) = \sup_{\mu \in {\mathbb M}_f} ( h_{\mu}(f) + \int \psi \, d \mu)  = $$ $$=
	 \sup_{\mu \in {\mathbb M}_f}  (h_{\mu}(f) - \int \log |\det (Df)| \, d \mu)  \leq 0.
	  $$
	So the pressure  $P(f, \psi) $ is not positive.
	
Now, recalling Definition \ref{DefinitionPesinFormula} of Pesin's entropy formula,
and using again Corollary \ref{CorollaryLiouvilleFormulaForExpandingMaps},   
an $f$-invariant measure $\mu$ satisfies Pesin's formula if and only if
$$ 0 = h_{\mu}(f) - \int \log |\det (Df)| \, d \mu =h_{\mu}(f) + \int \psi \, d \mu =
\sup _{\mu \in {\mathbb M}_f} (h_{\mu}(f) + \int \psi \, d \mu).  $$
\end{proof}

The main result to be proved along this paper is the  following:

\begin{maintheorem} 
	\label{MainTheorem}
	
	Let $f:M \mapsto M$ be an expanding $C^1$ map on a compact Riemannian manifold $M$ of finite dimension.
	
	Then, any (necessarily existing) pseudo-physical measure $\mu$ for $f$ satisfies Pesin's entropy formula. 	
		Namely,
	$$   h_{\mu}(f) =  \int \sum_{i=1}^{m(x)} \chi_i^+(x) \, d \mu = \int  \log |\det (Df)|\, d \mu. $$
	Equivalently, $\mu$ is an equilibrium state for the potential
	 $$\psi= -\log |\det (Df)|, $$ and the pressure $P(f, \psi) = 0$.

\end{maintheorem}

Theorem \ref{MainTheorem} is a generalization to any finite dimension  of the result previously proved in \cite{CatEnr-Portugaliae} in dimension one. 
The proof of Theorem \ref{MainTheorem}, which we will expose along this paper, was  presented by F. Valenzuela is his unpublished  thesis in 2017 \cite{Valenzuela-Tesis}.  In 2019, and previously in 2017 as a preprint,  Araujo and Santos  \cite{Araujo-Santos-2019} proved a more general result that holds not only for $C^1$ (uniformly) expanding maps of Theorem \ref{MainTheorem} (according to Definition \ref{DefinitionExpandingMaps}), but also for  maps that are non-uniformly expanding.

%%%%%%%%%%%%%%%%%%%%%%%%%%%%%

\vspace{.5cm}

\noindent {\bf Organization of the paper. }

In Section \ref{Section2-Expansiveness}, we prove an important topological property of expanding maps (the expansiveness) that we will need to prove Theorem \ref{MainTheorem}.
In Section \ref{SectionEntropy} we define the metric entropy and recall some of its well known properties.
In Section \ref{Section4-ProofOfMainTheorem}, we prove Theorem \ref{MainTheorem}.
In Section \ref{Section5-Examples}, we give examples.

\vspace{.5cm}

\noindent {\bf Acknowledgements.}
Sections  \ref{Section2-Expansiveness}, \ref{SectionEntropy} and  \ref{Section4-ProofOfMainTheorem} are a translation to English, with some light changes and adding,  of the Master Thesis of Fernando Valenzuela \cite{Valenzuela-Tesis} .  

Fernando Valenzuela thanks the financial support of the Master Scholarship during his Graduate Studies at PEDECIBA  (Programa de Desarrollo de Ciencias Básicas), Área Matemática (Uruguay).  Eleonora Catsigeras thanks the financial support of ANII (Agencia Nacional de Investigación e Innovación) and PEDECIBA, both in Uruguay.

\section {Expanding maps are expansive.} \label{Section2-Expansiveness}
To prove Theorem \ref{MainTheorem} we need an important topological property defined for continuous maps, called expansiveness. We need to show that, in particular, the $C^1$ expanding maps, accoding to Definition \ref{DefinitionExpandingMaps},  are expansive.

\begin{definition} {\bf (Expansive maps.)}
	\label{DefinitionExpansive} \em 
	
	A continuous map $f: M \mapsto M$ is \em expansive in the future \em, if there exists a constant $\alpha>0$, called the \em expansivity constant, \em  such that if $x,y \in M$ satisfy
	$$ \dist (f^n(x), f^n(y)) \leq \alpha \ \ \forall \ n\geq 0,      $$
	then $x=y$.
\end{definition}
The expansiveness is understood as the sensitivity to the initial condition. In fact, the two orbits with two different initial states, even if these initial states are arbitrarily near, they separate more than $\alpha$ for some iterate in the future.

 \begin{proposition} \label{PropositionExpandingIsExpansive}
	
	If the $C^1$ map $f : M \mapsto M$ on the compact Riemannian manifold $M$ is expanding, then it is 
	expansive in the future.
\end{proposition}

\begin{proof}
Let $x \in M$, and $\delta >0$. Denote by $B_{\delta}(0) \subset T_xM $ the open ball in the tangent space at $x$,  centered at $0$ with radius $\delta$, i.e. the set of vectors in $T_xM$ with norm smaller than $\delta$. Choose $\delta$ small enough such that the exponential map $\exp_x:B_{\delta}(0) \subset T_xM  \rightarrow M$ is a diffeomorphism onto its image.
	Explicitly, for any $y \in M$ such that $\mbox{dist}(x,y)<\delta$, there exists a unique vector $v=\exp^{-1}_x (y) \in T_xM$, and this vector satisfies $\| v \|=\mbox{dist}(x,y)<\delta_1$. Since $M$ is compact, we can choose a uniform $\delta >0$, namely $\delta$ does not depends on  $x$.
	
	In the sequel we will denote $y-x$ to refer to the vector $v=\exp_x^{-1}(y) \in B_{\delta}(0)  \subset T_xM$.
	
	Since $f$ is of $C^1$ class
	$$\|f(y)-f(x)\|=\|Df_z(y-x)\|,$$
for some point $z$ in the (convex) ball centered at $x$ of radius $\delta$.
	
Then $$\|f(y)-f(x)\| \geq \lambda \| y-x \|, $$
where $\lambda >1$ is the constant in Definition \ref{DefinitionExpansive} of the expanding map $f$.

	It is enough to prove that $\delta>0$ is a constant of expansivity for $f$, as in Definition \ref{DefinitionExpansive}.
	Assume that $$\dist (f^j(x), f^j(y)) = \| f^j(y) - f^j(x)\| \leq \delta \ \ \forall \ j \geq 0.$$
	Then $$ \| f^{j+1}(y) - f^{j+1}(x)\| \geq \lambda  \| f^j(y) - f^j(x)\|  \ \ \forall \ j \geq 0,$$
	and therefore
	$$ \delta \geq  \dist (f^n(x), f^n(y)) = \| f^{n}(y) - f^{n}(x)\| \geq \lambda^n  \|y-x\|  \ \ \forall \ n \geq 0.$$
Since $\lambda^n \rightarrow + \infty$ with $n$, while $\lambda^n\|y-x\|\leq \alpha$ for all $n \geq 0$, we deduce that $x= y$.
\end{proof}

\section{The metric entropy.} \label{SectionEntropy}

In this section we review the definition of metric entropy of $f$ with respect to an invariant measure $\mu$ and state some of its properties that we will use in the proof of Theorem \ref{MainTheorem}.

A \em finite measurable partition \em  $\mathcal{P} $ is a finite family of measurable
sets $P \subset M$ that are pairwise disjoint and whose union is $M$. The sets $P \in \mathcal P$ are \em the pieces of \em $\mathcal P$. We agree to simply name \em partition \em to refer to a finite measurable partition. 

The \em boundary \em $\partial \mathcal P$ of a partition is the union of the topological boundaries of its pieces. Namely,
$$\partial \mathcal P :=  \bigcup \{ \partial P: \ P \in \mathcal P\}.$$

The \em diameter \em $\mbox{diam} (\mathcal P)$ of a partition is the maximum diameter of its pieces. Namely,
$$\mbox{diam} (\mathcal P) = \max_{P \in \mathcal P} \mbox{diam}(P).    $$
The \em product \em ${\mathcal P} \vee {\mathcal Q} $ of two partitions is the new partition whose pieces are $P \bigcap Q$, where $P\in {\mathcal P}$ and $Q\in {\mathcal Q}$.

More generally, for each natural number $N \geq 1$, if $\{\mathcal P_i\}_{1 \leq i \leq N}$ is a collection of $N$ partitions ${\mathcal P}_i$, we define their \em product \em  as follows:
$$\bigvee_{i= 1}^N \mathcal{P}_i=\left \{ \bigcap_{i= 1}^N P_i \colon \ \ \ \ P_i \in \mathcal{P}_i  \ \ \forall \ 1 \leq i \leq N\right \}.$$

\begin{definition}
	\label{Definition H(P mu)} 
	\em The \em entropy $ H({\mathcal P}, \mu) $ \em of the partition  ${\mathcal P}$ with respect to a probability measure $\mu$  is
		$$ H({\mathcal P}, \mu) : =-\sum _{P \in {\mathcal P}} \mu (P) \log \mu (P),
	$$
	where we
agree to take $0 \cdot\log 0 = 0$.
	
\end{definition}

Now, let us introduce the dynamics of $f: M \mapsto M$ to the study of the entropy of the partitions with respect to a probability measure.

For any partition $\mathcal P$, we consider the following product partition:
\begin{equation}
	\label{Equality Defining P^n}
	{\mathcal P}_f^n := \bigvee_{i=0}^{n-1} f^{-i} \mathcal{P}={\mathcal P}\vee f^{-1}{\mathcal P}\vee \ldots \vee f^{-(n-1)}
	{\mathcal P},\end{equation}
where $f^{-i} \mathcal{P}:=\{ f^{-i} (P): P \in \mathcal{P} \}$.

\begin{proposition} If $\mu$ is $f$-invariant, then the following limit exists and  satisfies the equality and inequality at right:
	$$ \lim  _{n \rightarrow \infty} \frac{ H({\mathcal P}_f^n, \mu)}{n} = \inf _{n \geq 1} \frac{ H({\mathcal P}_f^n, \mu)}{n} \leq H({\mathcal P}, \mu).$$
\end{proposition}
\begin{proof}
See for instance \cite{Viana-Oliveira-Book}, Lemma 9.1.12.
\end{proof}

\begin{definition} \em \label{DefinitionEntropyOf f w.r.partition}
	Let $\mathcal P$ be a partition, and $\mu$ be an $f$-invariant measure,
	We call the following expression $ h_\mu(f,{\mathcal P})$ \em the entropy of $f$ with respect to the partition $\mathcal P$ and to the measure $\mu$: \em
	$$ h_{\mu}(f,{\mathcal P}) := \lim  _{n \rightarrow \infty} \frac{ H({\mathcal P}_f^n, \mu)}{n} = \inf _{n \geq 1} \frac{ H({\mathcal P}_f^n, \mu)}{n}.$$
	\em \end{definition}

\begin{definition} {\bf (The metric entropy.)}
	\label{DefinitionMetricEntropy}
		\em	
	
Let $\mu$ be an $f$-invariant probability measure. We call the following expression $h_{\mu}(f)$ \em the metric entropy of $f$ with respect to the measure $\mu$:
	$$ h_\mu (f) \ :=\  \sup \ \{ h(f,{\mathcal P}): \   {\mathcal P} \mbox{ finite measurable partition} \}.$$

\end{definition}

\subsection{Properties of the entropy of partitions.}

In this subsection we state some properties of the entropy of partitions that will be used in the proof of Theorem \ref{MainTheorem}.  For more properties of the entropy, see for instance the books \cite{kellerBook} and
\cite{Viana-Oliveira-Book}.

\begin{proposition}
	\label{PropositionH leq log p}
	For any (finite measurable) partition ${\mathcal P}$ and any probability measure $\mu$
$$	H({\mathcal P}, \mu)  \leq \log p,$$
	where $p$ is the number of pieces of $\mathcal P$.
\end{proposition}
\begin{proof} See for instance \cite{Viana-Oliveira-Book}, Lemma 9.1.3. \end{proof}

\begin{proposition}\label{Proposition H(P vee Q)}.
	Let ${\mathcal P, Q }$ be two partitions and $\mu$ any probability measure. Then,	
	$$ H({\mathcal P},\mu) \leq H({\mathcal P \vee Q},\mu) \leq H({\mathcal P,\mu}) + H({\mathcal Q},\mu).$$
\end{proposition}
\begin{proof}   See for instance \cite{Viana-Oliveira-Book}, Section 9.1.
	%\cite{Catsigeras-TeoErg-Book}, Corollary 9.1.8.
 \end{proof}

\begin{proposition}  \label{Proposition H(P^n) H(P, f*mu)}  For any partition $\mathcal P$ and any  (not necessarily $f$-invariant) probability measure $\mu$:
	$$H({\mathcal P}_f^n, \mu) \leq \sum_{i=0}^{n-1}H(f^{-i} \mathcal P, \mu) = \sum_{i=0}^{n-1}H(\mathcal P, {f^*}^{i} \mu).   $$	
\end{proposition}

\begin{proof}
	
	Applying Equality (\ref{Equality Defining P^n}) and Proposition \ref{Proposition H(P vee Q)} we have
	\begin{equation}
		\label{Equality2}
		H({\mathcal P}_f^n, \mu) = H\left( \bigvee_{i= 0}^{n-1} f^{-i} \mathcal P, \mu \right) \leq \sum_{i=0}^{n-1}H(f^{-i} \mathcal P, \mu).
	\end{equation}
	Besides, from Definition \ref{Definition H(P mu)}:
	$$H(f^{-i} \mathcal P, \mu) = -\sum_{P \in \mathcal P} \mu(f^{-i} P) \log  \mu(f^{-i} P),    $$
	and from the definition of the pull-back $f^*$, we have
	$ \mu (f^{-i} P) =  {f^*}^{i} \mu(P) $. So,
	$$H (f^{-i} \mathcal P, \mu) = H(\mathcal P, {f^*}^{i} \mu) .$$
	Finally, substituting this last equality in (\ref{Equality2}), we deduce
	$$H({\mathcal P}_f^n, \mu) \leq \sum_{i=0}^{n-1}H(f^{-i} \mathcal P, \mu) = \sum_{i=0}^{n-1}H(\mathcal P, {f^*}^{i} \mu),  $$ 
	as wanted. 
\end{proof}

\begin{proposition}
	\label{PropositionEntropyConvexComb}
	For any partition $\mathcal P$ and any (not necessarily $f$-invariant) probability measure $\mu$,
	$$\frac{H({\mathcal P}_f^n, \mu)}{n} \leq \frac{1}{n} \sum_{i=0} ^{n-1} H( \mathcal P,  {f^*}^i\mu) \leq H \left( \mathcal P,  \frac{1}{n} \sum_{n=0}^{n-1} {f^*}^i\mu    \right).    $$ 
\end{proposition}

\begin{proof}
	Consider the following  continuous real function $\phi: [0,1] \mathbb{R}$:
	$$\phi(u ) = -u\log u \mbox{ if } u \in(0,1], \ \ \phi(0)= 0.$$
	It is easy to check that $\phi$ is $C^{\infty}$ in $(0,1)$, and that $\phi''(u) <0$ for all $u \in (0,1)$.
	Therefore, the graph of $\phi$ is above the secant line. Thus, the value of $\phi $ at the convex combination of $n$ values $u_0, \ldots, u_{n-1} \in [0,1]$ (which is the ordinate of a point in the graph of $\phi$), is larger or equal than the convex combination of 
	$\phi(u_0), \ldots, \phi(u_{n-1})$ (which is the ordinate of a point in the secant line). Precisely, if $0<\lambda_i <1 $ for all $0\leq i\leq n-1$ and  $ \sum_{i=0} ^{n-1}\lambda _i =1$, then
	$$ \phi\left( \sum_{i=0} ^{n-1}\lambda _i u_i \right) \geq  \sum_{i=0} ^{n-1}\lambda _i \phi(u_i). $$
	Therefore
	$$ H \left( \mathcal P,  \frac{1}{n} \sum_{n=0}^{n-1} {f^*}^i\mu\right)  = \sum_{P \in \mathcal P}  \phi\left( \frac{1}{n} \sum_{i=0} ^{n-1} {f^*}^i\mu (P)  \right) \geq $$
	$$   \sum_{P \in \mathcal P}  \frac{1}{n} \sum_{i=0} ^{n-1} \phi ({f^*}^i\mu (P)) =   $$ 
	$$ \frac{1}{n} \sum_{i=0} ^{n-1} \sum_{P \in \mathcal P} \phi ({f^*}^i\mu (P)) =  
	\frac{1}{n} \sum_{i=0} ^{n-1} H( \mathcal P,  {f^*}^i\mu).	
	$$ 
	Finally, using Proposition \ref{Proposition H(P^n) H(P, f*mu)} the last expression is greater or equal than
	$(1/n)	H({\mathcal P}_f^n, \mu) $, as wanted.	
\end{proof}

\begin{proposition}  \label{PropositionH(P,mu)ContinuaEn mu}
	
	Let $\mathcal P$ be a (finite measurable) partition and $\mu$ a probability measure such that
	$$\mu(\partial P) =0.$$
	If $\{\mu_n\}_{n\geq 0}$ is a sequence of probabilities measures such that
	$$ {\lim}^*_{n\rightarrow + \infty} \ \mu_n = \mu,  $$
	then
	$$\lim_{n \rightarrow + \infty} H({\mathcal P}, \mu_n) = H({\mathcal P}, \mu).$$
	
\end{proposition}
In other words, this proposition states the continuity at $\mu$ of the entropy of a partition $\mathcal P$ as a function of the measure $\mu$, if the boundary of the partition has zero $\mu$-measure.

To prove Proposition \ref{PropositionH(P,mu)ContinuaEn mu}, we will use the following lemma:

\begin{lemma}
	\label{Lemma}
	Let $\mu_n$, $\mu$ be probability measures such that 
	$${\lim}^*_{n \rightarrow + \infty} \ \mu_n = \mu$$

	\noindent {\em (1)} If $K \subset M$ is compact,  then $\limsup_{n \rightarrow + \infty} \mu_n (K) \leq \mu(K)$.
	
		\noindent {\em (2)} If $V \subset M$ is open,  then $\liminf_{n \rightarrow + \infty} \mu_n (V) \geq \mu(V)$.
		
			\noindent {\em (3)} If $A$ is a Borel set such that $\mu(\partial A) = 0$, then $$\lim_{n \rightarrow + \infty} \mu_n (A) = \mu(A).$$
	
\end{lemma}
\begin{proof}
	\noindent {(1)} Let $\epsilon >0$ and $V \subset M$ such that $K \subset V$ and $\mu(V \setminus K)<\epsilon$. Let $\phi:M \mapsto [0,1]$ be a continuous function such that  $\phi|_K=1$ and $\phi|_{M \setminus V}=0$. Then
		$$\mu_n(K) \leq \int \phi d\mu_n  \ \ \forall \ n \geq 1.$$
		From the continuity of $\phi$, and the convergence in the weak$^*$ topology of $\mu_n$ to $\mu$, we obtain
		$$\int \phi d\mu_n \rightarrow \int \phi d\mu  = \int _{\phi \neq 0} \phi \, d \mu \leq \int _{V} \phi \, d\mu \leq \int _V 1 \, d \mu =$$ $$= \mu(V) = \mu(K)+\mu(V \setminus K)<\mu(K)+\epsilon.$$
		Luego $$\mu_n(K) \leq \int \phi \, d \mu_n <  \mu(K) + \epsilon \ \ \ \forall \ n \geq 1. $$   Now, taking $\limsup$ we obtain $$ \limsup_{n \rightarrow + \infty} \mu_n(K) \leq \mu(K) + \epsilon.$$
		Since the above inequality holds for all $\epsilon >0$, we conclude
		$$\limsup_{n \rightarrow + \infty} \mu_n(K) \leq \mu(K), $$
	as wanted.
	
		\noindent {(2)} Let $K=M \setminus V$. We have
		$$\mu_n(V)=1-\mu_n(K) \ \ \forall \ n \geq 1.$$
		Taking $\liminf$, we obtain
		$$\liminf_{n \rightarrow + \infty} \mu_n(V)=1-\limsup_{n \rightarrow + \infty} \mu_n(K). $$
	Since $K$ is a closed set in the compact metric space $M$, it is compact. Applying property (1), we conclude  $$\liminf_{n \rightarrow + \infty} \mu_n(V)=1-\limsup \mu_n(K) \geq 1-\mu(K)=\mu(V), $$
		as wanted.
		
		\noindent {(3)}  We consider the interior $\mbox{int}(A)$ of $A $ and its closure $\overline{A}$. Each one of these sets differs from $A$ i the boundary of $A $, which has zero $ \mu$-measure. Thus
		$$\mu(\mbox{int}(A))=\mu(A)=\mu(\overline{A}).$$ For $\mu_n$ we have $$\mu_n(\mbox{int}(A)) \leq \mu_n(A) \leq \mu_n(\overline{A}).$$
		Applying properties (1) and (2), we obtain
		$$\mu(\mbox{int}(A)) \leq \liminf_{n \rightarrow + \infty} \mu_n(\mbox{int}(A))  \leq \liminf_{n \rightarrow + \infty} \mu_n(A) \leq \limsup_{n \rightarrow + \infty} \mu_n(A) $$  $$ \leq \limsup_{n \rightarrow + \infty} \mu_n(\overline{A}) \leq \mu(\overline{A}).$$
		Since $\mu(\mbox{int}(A))=\mu(\overline{A}) = \mu(A)$, the last chain of inequalities is a chain of equalities. Therefore, $\limsup $ and  $\liminf$ of $\mu_n(A)$ coincide and are equal to $\mu(A)$. We conclude
		$$\lim_{n \rightarrow + \infty} \mu_n(A) = \mu(A),$$
	as wanted.
\end{proof}

\begin{proof} {\em of Proposition \ref{PropositionH(P,mu)ContinuaEn mu}}:
	
Consider the following continuous real function  $\phi:[0,1] \mapsto \mathbb{R}$:
	$$\phi(u)=- u \log u  \mbox{ if } 0 < u \leq 1, \ \ \phi(0)= 0.$$
	From Definition \ref{Definition H(P mu)} we have
	$$H(\mathcal{P},\mu)=\sum_{P \in \mathcal P} \phi(\mu(P)), \ \  H(\mathcal{P},\mu_n)=\sum_{P \in \mathcal P} \phi(\mu_n(P))    $$
	
	Since $\mu(\partial \mathcal{P}) = 0$ we can apply part (3) of Lemma \ref{Lemma}  to each piece $P \in {\mathcal P}$. Using also the continuity of the function $\phi$, we obtain
	
	$$\lim_{n \rightarrow + \infty} \phi(\mu_n(P)) = \phi(\mu(P) )\mbox{   } \ \ \forall \ P \in \mathcal P.$$
	
Finally we sum the above equality for all the finite number of pieces $P$ of the partition $\mathcal P$, concluding that	
	$$\lim_{n \rightarrow + \infty} H(\mathcal{P},\mu_n)= \lim_{n \rightarrow + \infty} \sum_{P \in \mathcal P}  \phi (\mu_n(P))= \sum_{P \in \mathcal P}  \lim_{n \rightarrow + \infty} \phi (\mu_n(P))=$$
	$$=\sum_{P \in \mathcal P}  \phi (\mu(P))=H(\mathcal{P},\mu),$$ as wanted.
\end{proof}

\subsection{The metric entropy for expansive maps.}

For expansive maps, Kolmogorov-Sinai Theorem (Theorem \ref{Theorem Kolm-Sinai} and Corollary \ref{CorollaryKolm-Sinai For Expansive maps} in this section) states there exists partitions that reach the  supremum in Definition \ref{DefinitionMetricEntropy}. In other words, the metric entropy of $f$ may be computed as the entropy of $f$ with respect to a concrete partition.

\begin{definition}	
	\em A partition ${\mathcal P} $ of $M$,  is a  \em generator \em (in the future) for $f:M \to M$  if the $\sigma$-algebra generated by $\displaystyle \left\{ \bigvee _{j =0}^{k} f^{-j} {\mathcal P} \right\}_{k\geq 0} $ is the Borel $\sigma$-algebra.
\end{definition}
Recall Definitions \ref{DefinitionEntropyOf f w.r.partition} and \ref{DefinitionMetricEntropy} of $ h_\mu(f, {\mathcal P})$ and  $h_\mu (f) $.
\begin{theorem}{\bf (Kolmogorov-Sinai)}
	\label{Theorem Kolm-Sinai}
	
	If the (finite measurable) partition ${\mathcal P} $ of $M$ is a generator for  $f:M \to M$, then for any $f$-invariant probability measure $\mu$:
	$$h_\mu (f) = h_\mu(f, {\mathcal P}).$$
\end{theorem}
\begin{proof}
	See for instance %\cite{Walters-Book}, Theorem 9.2.1, or 
	\cite{kellerBook}, Theorem 3.2.18.
\end{proof}
Recall Definition \ref{DefinitionExpansive} of expansiveness in the future of a map. 
\begin{corollary}
	\label{CorollaryKolm-Sinai For Expansive maps}
	If $f: M \mapsto M$ is expansive in the future with expansivity constant $\alpha$, and $\mathcal P$ is a partition with $\mbox{diam}(\mathcal P) <\alpha$ then  for any $f$-invariant probability measure $\mu$:
	$$h_\mu (f) = h(f, {\mathcal P})= \lim_{n\rightarrow +\infty} \frac{H({\mathcal P}^n_f, \mu)}{n}.$$
\end{corollary}

\begin{proof}
	
	Applying Theorem \ref{Theorem Kolm-Sinai} it is enough to prove that  $\mathcal{P}$ is a generator for $f$.

	Let $k \geq 0$ and for each point $x \in M$ consider the piece $A_k(x) $ of the partition $\bigvee^k_{i=0} f^{-i} \mathcal{P}$ that contains $x$. We denote by $B_\delta(x)$ the ball centered at $x$ with radius $\delta<0$. 
	We will first prove the following statement:

	\noindent {\bf Assertion A. }
\em 	For all $0<\delta<\alpha$ there exists $k=k(\delta)$ such that	 \em 
	\begin{equation}
		\label{EquationA_k subset B}
		A_k(x) \subset B_{\delta} (x)\ \ \ \forall \  x\in M.
	\end{equation}

Suppose for a contradiction that there is $\delta \in (0, \alpha)$ such that for all $k \geq 0$ there exist points $x_k, y_k \in  M$ satisfying
$$y_k \in A_k(x_k) \setminus B_\delta(x_k) \ \ \forall \ k \geq 0.$$
Since $M$ is compact, there are subsequences $\{x_{k_j}\}_j $ and  $ \{y_{k_j}\}_j   $ convergent to the points $x$ and $y$ respectively. On the one hand, as $y_k \not \in B_\delta(x_k)$ we have
\begin{equation}
	\label{equality>delta}
	\dist(x,y) = \lim_{j\rightarrow + \infty} \dist(x_{k_j},  y_{k_j}  ) \geq \delta.
\end{equation}
On the other hand, as $y_k \in A_k(x_k) \in \bigvee_{i=0}^k f^{-i} \mathcal P   $ and $\mbox{diam}(\mathcal P) < \alpha$, we obtain
$$\dist(f^i(x_{k_j}),   f^i(y_{k_j})   )  < \alpha \ \ \forall \ 0 \leq i \leq k_j.  $$
Taking the limit in the inequality above when $j \rightarrow +\infty$ with $i$ fixed, we deduce $$\dist(f^i(x),   f^i(y)   )  \leq \alpha \ \ \forall \ i \geq 0.  $$
Due to the expansiveness in the future of $f$, the inequality above implies $x=y$ contradicting inequality (\ref{equality>delta}), and ending the proof of Assertion A.

Any open set $V \subset M$ can be written as the union of  open balls $B_{\delta (x)}(x)\subset V$ with $\delta (x)< \alpha$. Using Assertion A, each point $x$ of $V$ is inside a piece  $A_{k(x)}(x)\subset B_{\delta }(x)$ for some $ k(x) \geq 0$. Therefore $V$ is the union of a family of  pieces    $$A_{k(x)}(x) \in \bigcup_{k \geq 0} \left\{ \bigvee _{i =0}^{k} f^{-i} {\mathcal P}: \ k \geq 0  \right\}.$$ Since the family of all such pieces  is countable (because they are the pieces of a countable union of finite partitions), we deduce that the $\sigma$-algebra generated by them includes all the open subsets of $M$. Thus, it includes the Borel $\sigma$-algebra.   Conversely,  all the pieces $A_k(x)$ are Borel measurable sets. Therefore, the  $\sigma$-algebra generated by them is included in the  Borel $\sigma$-algebra. We conclude that both  $\sigma$-algebras coincide, as wanted.
\end{proof}

\section{Proof  of Theorem \ref{MainTheorem}.} \label{Section4-ProofOfMainTheorem}

To prove Theorem \ref{MainTheorem}, we will fix some notation:

\noindent For the $C^1$ expanding map $f:M \mapsto M$ denote
 \begin{equation}
	\label{Equality Def psi} 
	\psi (x) := - \log |\mbox{det}Df_x| < 0, \ \ \forall \ x \in M.
\end{equation}
Recall that $\psi: M \to  \mathbb{R}$ is the potential in the statement of Theorem \ref{MainTheorem}.

 In the space ${\mathcal M}$  of  Borel probability measures on the manifold $M$, we fix the following weak$^*$ metric:

\begin{equation}
	\label{EqualityDef weak^*distance}
	\mbox{dist}^*(\mu, \nu) := \sum_{i= 0}^{+ \infty} \frac{1}{2^i} \; \left |\int \phi_i \, d \mu - \int \phi_i \, d \nu\right |, \end{equation}
where $\phi_{\,0}:= \psi$ and $\{\phi_i\}_{i \geq 1}$ is a countable family of continuous functions   that is dense in the space $C^0(M, [0,1])$.

	\begin{lemma}
		\label{LemmaBolasConvexas}
		 For any $\mu \in {\mathcal M}$ and any $\epsilon >0$ the ball ${\mathcal B} := \{\nu \in {\mathcal M}: \ \mbox{dist}^* (\mu, \nu) < \epsilon\}$   is convex.
	\end{lemma}
	
	\begin{proof}
		
		Let $\nu_1$, $\nu_2\in \mathcal{B}$.
		We will prove that if $\lambda_1,\lambda_2 \in [0,1]$ are such that $\lambda_1 + \lambda_2 =1$, then $\lambda_1 \nu_1 + \lambda_2 \nu_2 \in \mathcal{B}.$

		Using the triangle inequality, we have
		$$\mbox{dist} (\lambda_1 \nu_1 +\lambda_2 \nu_2, \mu)=\sum^{+\infty}_{i=0} \frac{1}{2^i}\left| \int \phi_i d\mu - \int \phi_i d(\lambda_1 \nu_1 +\lambda_2 \nu_2) \right| $$ $$\leq \sum^{+\infty}_{i=0} \frac{1}{2^i}
		 \left | \lambda_1 \int \phi_i d\mu - \lambda_1 \int \phi_i d\nu_1 \right| + \sum^{+\infty}_{i=0} \frac{1}{2^i} 
		 \left | \lambda_2 \int \phi_i d\mu - \lambda_2 \int \phi_i d\nu_2 \right | = $$ $$  = \lambda_1 \sum^{+\infty}_{i=0} \frac{1}{2^i} \left| \int \phi_i d\mu - \phi_i d \nu_1 \right | + \lambda_2 \sum^{+\infty}_{i=0} \frac{1}{2^i} \left| \int \phi_i d\mu - \phi_i d \nu_2 \right | <$$ $$ (\lambda_1 + \lambda_2) \epsilon = \epsilon,$$
		as wanted.		
\end{proof}

We state and prove now a series of lemmas that we will use in the proof of Theorem \ref{MainTheorem}.

Recall  Definition \ref{Definition H(P mu)} of the entropy $H({\mathcal P}, \mu)$ of a partition with respect to a probability measure $\mu$, and the equality (\ref{Equality Defining P^n}) defining the product partition ${\mathcal P}^n_f$.
\begin{lemma} \label{lemma0} 
	
	Let $\mu$ be an $f$-invariant probability measure. Let ${\mathcal P}$ be any finite partition of the manifold $M$ into measurable sets. If  $\mu(\partial {\mathcal P}) = 0,$  then, for all $\epsilon > 0$, and for all $q \geq 1$ there exists $\epsilon^* > 0$ such that \begin{equation} \label{eqn22}
		\Big|\frac{H({\mathcal P}_f^q, \rho)}{q} - \frac{H({\mathcal P}_f^q, \mu)}{q} \Big| < \epsilon\end{equation} for any probability measure $\rho$ such that
	$\mbox{dist}(\rho, \mu) < \epsilon^*.$
\end{lemma}
\begin{proof}
	Assume by contradiction that for all $\epsilon^* >0$ there exists a probability measure $\rho$, whose distance to $\mu$ is smaller than $\epsilon^*$, and that does not satisfy inequality (\ref{eqn22}). Thus, in particular for $\epsilon^*= 1/m$  where $m \in \mathbb{N}$, there exists $\rho_m$  such that
	$$\mbox{dist}(\rho_m, \mu) < \frac{1}{m},$$
	\begin{equation} \label{eqn23}\Big|\frac{H({\mathcal P}_f^q, \rho_m)}{q} - \frac{H({\mathcal P}_f^q, \mu)}{q} \Big| \geq \epsilon  \ \ \forall  \ m \geq 1.\end{equation}
	We have $\lim_{m \rightarrow + \infty} \rho_m = \mu$ in the weak$^*$ topology and $\mu(\partial {\mathcal P}) = 0$. Note that  $\mu(\partial {\mathcal P}_f^q)= 0$ because $\mu$ is $f$-invariant. So,  applying part (iii) of Lemma \ref{Lemma},  $\lim_{m \rightarrow + \infty} \rho_m (Y) = \mu(Y)$ for any piece $Y \in {\mathcal P}_f^q$. And, from Proposition \ref{PropositionH(P,mu)ContinuaEn mu}, for fixed $q \geq 1$, we have:
	$$\lim _{m \rightarrow + \infty} \frac{H({\mathcal P}_f^q, \rho_m)}{q} = \frac{H({\mathcal P}_f^q, \mu)}{q},$$
	contradicting inequality (\ref{eqn23}).
\end{proof}

Recall Definition \ref{DefinitionExpansive} of expansiveness, and Proposition \ref{PropositionExpandingIsExpansive}. Let $\alpha >0$ be an expansivity constant for $f$. From Corollary \ref{CorollaryKolm-Sinai For Expansive maps} of Kolmogorov-Sinai Theorem for expansive maps, for any $f$-invariant probability measure $\mu$ we have:
\begin{equation}
	\label{equationEntropy}
	h_{\mu} (f)= \lim_{q \rightarrow + \infty} \frac{H({\mathcal P}_f^q , \mu)}{q} \ \ \mbox{ if } \ \  \mbox{diam}({\mathcal P}) < \alpha.
\end{equation}

\begin{lemma} 	
	\label{Lemma1} Let $f: M \mapsto M $ be a $C^1$ expanding map on the compact manifold $M$. Let $\alpha >0$ be an expansivity constant for $f$.
	
	For all $0 < \delta < \alpha$, for all $\epsilon > 0$, and for any $f$-invariant measure $\mu$, there exists a finite partition ${\mathcal P}$ of $M$,  a real number $\epsilon^* >0$, and a natural number $n_0 \geq 1$, such that:
	
	\noindent \em (i)   \em  $\mbox{\em diam}({\mathcal P}) < \delta < \alpha  $,
	
	\noindent \em (ii) \em $\mu (\partial {\mathcal P} ) = 0$,
	
	\noindent \em (iii) \em For any sequence of non necessarily invariant probabilities $\nu_n$,
	if    $\mu_n:= \frac{1}{n} \sum_ {j= 0}^{n-1} (f^j)^* \nu_n,$  and if
	$\mbox{\em dist} (\mu_n, \mu) < \epsilon^* \ \ \forall \ n \geq 1,$
	then
	$$ \frac{1}{n } H({\mathcal P}_f^{n }, \nu_{n }) \leq h_{\mu}(f) + \epsilon \ \ \ \forall \ n \geq n_0.$$

\end{lemma}

\begin{proof} Recall Equality (\ref{Equality Defining P^n}) defining the product partition ${\mathcal P}_f^{n }  $ and Definition \ref{DefinitionMetricEntropy} of the metric entropy $h_{\mu} (f)$ of $f$. For simplicity along this proof, since we will not change the map $f$,  we will denote 	$h_{\mu} $ instead of $ 	h_{\mu} (f)$ and
${\mathcal P}^{n }  $ instead of ${\mathcal P}_f^{n }  $.	

 Take any finite covering ${\mathcal U} = \{Y_1, \ldots, Y_p\}$ of $M$ with open balls with radia smaller than $\delta/2$.  Denote $\partial{\mathcal U} := \bigcup_{i=1} ^p \partial Y_i  $.  Since the family of boundaries of the balls with radius $r>0$ is non countable when changing $r$,  but these boundaries can have positive $\mu$-measure only for at most a countable subfamily, the radius of each ball $Y_i \in {\mathcal U}$ can be chosen such that $\mu(\partial Y_i) = 0$. Thus $\mu (\partial {\mathcal U}) = 0$.    Therefore, the partition ${\mathcal P} = \{X_i\}_{1 \leq i \leq p}$ defined by $X_1 := Y_1 \in {\mathcal U} $, \  $X_{i+1} :=  Y_{i+1} \setminus (\cup_{j= 1}^i X_i)$, satisfies the assertions (i) and (ii).

To end the proof, for any given $\epsilon > 0$ let us find $\epsilon^* > 0$ and $n_0 \geq 1$ such that assertion (iii) holds.

Let us fix two integer numbers $q \geq 1$   and $n \geq q$. Write $n= N q + j$ where $N, j$ are integer numbers such that $ 0 \leq j \leq q-1$ Fix a (non necessarily invariant) probability $\nu$. Applying Propositions \ref{Proposition H(P vee Q)} and \ref{Proposition H(P^n) H(P, f*mu)}, we obtain:

$$H({\mathcal P}^n, \nu) = H({\mathcal P}^{ Nq+ j }, \nu)   \leq $$ $$ H( \vee_{i= 0}^{j-1}  f^{-(Nq+i)}{\mathcal P}, \nu) + H({\vee_{i= 0}^{N-1}f^{-iq}\mathcal P}^{q }, \nu)  \leq $$
$$\sum_{i= 0}^{j-1} H(f^{-(Nq+i)} {\mathcal P}, \nu) + \sum_{i= 0}^{N-1} H(f^{-iq}{\mathcal P}^{q },  \nu) = \sum_{i= 0}^{j-1} H({\mathcal P} , (f^{Nq+i})^* \nu) + \sum_{i= 0}^{N-1}H({\mathcal P}^{q }, (f^{iq})^*\nu).   $$

%\vspace{-.7cm}

From the above inequality, using Proposition \ref{PropositionH leq log p}, and recalling that $j \leq q-1<q$, we obtain
$$ H({\mathcal P}^n, \nu) \leq q \log p + \sum_{i= 0}^{N-1} H({\mathcal P}^{q }, (f^{iq})^*\nu) \ \ \forall \ q \geq 1, \ \ n \geq q,  $$
where $p$ is the number of pieces of the partition $\mathcal P$.

 The inequality above holds also for $f^{-l}{\mathcal P}$ instead of ${\mathcal P}$, for any $l \geq 0$, because it holds for any partition with exactly $p$ pieces.  Thus:
$$\ H(f^{-l}{\mathcal P}^n, \nu) \leq q \log p + \sum_{i= 0}^{N-1}H(f^{-l}{\mathcal P}^{q }, (f^{iq})^*\nu) =$$
$$ q \log p + \sum_{i= 0}^{N-1} H({\mathcal P}^{q }, (f^{iq + l})^*\nu). $$

\noindent Adding the above inequalities   for $0 \leq l \leq q-1$, we obtain:
$$\sum_{l= 0}^{q-1} H(f^{-l}{\mathcal P}^n, \nu) \leq  q^2 \log p + \sum_{l= 0}^{q-1} \sum_{i= 0}^{N-1} H({\mathcal P}^{q }, (f^{iq + l})^*\nu) $$
Therefore,  on the one hand we have:
\begin{equation} \label{equation0} \sum_{l= 0}^{q-1} H(f^{-l}{\mathcal P}^n, \nu) \leq q^2 \log p + \sum_{i= 0}^{Nq -1} H({\mathcal P}^{q }, (f^{i })^*\nu). 
\end{equation}
On the other hand, applying Proposition \ref{Proposition H(P vee Q)}, for all $0 \leq l \leq q-1$ we have
$$ H({\mathcal P}^n, \nu) \leq H ({\mathcal P}^n \vee f^{-n}{\mathcal P} \vee f^{-(n+1)}{{\mathcal P}}
\vee \ldots \vee f^{-(n+l-1)}{{\mathcal P}}, \nu) = H({\mathcal P}^{n+ l}, \nu) = $$ $$
 H( (\vee_{i=0}^{l-1} f^{-i} \mathcal P) \vee (f^{-l} {\mathcal P}^n) ).$$ Therefore, 
$$H({\mathcal P}^n, \nu) \leq H({\mathcal P}^{n+ l}, \nu) \leq \Big (\sum_{i= 0}^{l-1} H(f^{-i}{\mathcal P}, \nu) \Big) + H(f^{-l} {\mathcal P}^n, \nu).   $$
So,
 $$H({\mathcal P}^n, \nu)  \leq q\log p + H(f^{-l} {\mathcal P}^n, \nu).$$

%\vspace{-.2cm}

\noindent Adding the above inequalities   for $0 \leq l \leq q-1$ and joining with the inequality (\ref{equation0}), we obtain:

%\vspace{-.5cm}

$$q H({\mathcal P}^n, \nu) \leq q^2 \log p +  \sum_{i=0}^{q-1} H(f^{-l} {\mathcal P}^n, \nu) \leq 2 q^2 \log p + \sum_{i= 0}^{Nq -1} H({\mathcal P}^{q },
 (f^{i })^*\nu).$$
Recall that $n = Nq + j$ with $0 \leq j \leq q-1$. So $n-1 = Nq + j - 1 \geq Nq - 1$ and then

%\vspace{-.5cm}

$$q H({\mathcal P}^n, \nu)  \leq 2 q^2 \log p \ + \sum_{i= 0}^{n -1} H({\mathcal P}^{q }, (f^{i })^*\nu) $$

%\vspace{-.5cm}

\noindent  Now we put $\nu= \nu_n$ and divide by $n$. 
Using that $\mu_n = (1/n) \sum_{j=0} ^{n-1} {(f^j)} ^* \nu_n$ and applying Proposition \ref{PropositionEntropyConvexComb}, we obtain
$$ \frac{q\, H({\mathcal P}^n, \nu_n)}{n} \leq  \frac{2 q^2 \log p}{n} + \frac{1}{n} \, \sum_{i= 0}^{n -1} H({\mathcal P}^{q }, (f^{i})^*\nu_n)   \leq \frac{2 q^2 \log p}{n} + H({\mathcal P}^{q }, \mu_n ) .$$

%\vspace{-.2cm}

\noindent For any fixed $\epsilon >0$ (and the natural number $q \geq 1$ still fixed), take    $n \geq n(q):= \max\{q, 6\, q  \log p/ \epsilon \}$ in the inequality above. We deduce:

%\vspace{-.3cm}

$$\frac{q}{n} H({\mathcal P}^n, \nu_n) \leq  \frac{q \epsilon}{3} + H({\mathcal P}^q , \mu_n) \ \ \ \ \forall \ n \geq n(q)  \ \ \ \forall \ q \geq 1,$$
from where we obtain
\begin{equation}
	\label{equationSuper}
\frac{1}{n} H({\mathcal P}^n, \nu_n) \leq    \frac{\epsilon}{3} + \frac{H({\mathcal P}^q, \mu_n)}{q} \ \ \ \ \ \ \forall \   n \geq n(q) \ \ \ \ \forall \ q \geq 1.\end{equation}

%\vspace{-.2cm}

\noindent The inequality above holds for  for any fixed $q \geq 1$ and for any $n$ large enough, depending on $q$.

By hypothesis, $\mu$ is   $f$-invariant. So, after Equality (\ref{equationEntropy}), there exists $q \geq 1$ such that
\begin{equation}
	\label{equationSuper2}
	\frac {H({\mathcal P}^q, \mu)}{q} \leq  h_{\mu} + \frac{ \epsilon}{3}.\end{equation}

\noindent Fix such a value of $q$. Since $\mu(\partial({\mathcal P})) = 0$ due to the construction of ${\mathcal P}$ (depending on the given measure $\mu$), we can apply Lemma \ref{lemma0} to find $\epsilon^*>0$ such that
$$ \frac{ H({\mathcal P}^q, \rho)}{q} \leq \frac{H({\mathcal P}^q, \mu)}{q} + \frac{\epsilon}{3} \mbox{ if } \mbox{dist}(\rho, \mu) < \epsilon^*. $$

To prove assertion (iii) we assume $\mbox{dist} (\mu_n, \mu) < \epsilon^*$ for all $n \geq 1$. We deduce
$$ \frac{H({\mathcal P}^q, \mu_{n}) }{q} \leq \frac{H({\mathcal P}^q, \mu)}{q} + \frac{\epsilon}{3} \ \ \forall \ n \geq 1.$$
Joining this latter assertion with inequalities (\ref{equationSuper}) and (\ref{equationSuper2})  we obtain
$$ \frac{1}{n } H({\mathcal P}^{n }, \nu_{n }) \leq h_{\mu} + \epsilon \ \ \ \forall \ n \geq n(q).$$
Thus, after denoting $n_0:= n(q)$, assertion (iii) is proved.
\end{proof}

\noindent{\bf Notation.}  Recall Equality (\ref{Equality Def psi}) defining the continuous  real function $\psi: M \to \mathbb{R}$, which is the potential in the statement of Theorem \ref{MainTheorem}.
	For any real number $r \geq 0$ construct
	
	%\vspace{-.7cm}
	
	\begin{equation}\label{equationKr}{\mathcal K}_r := \{\nu \in {\mathcal M}_f: \;
		\int \psi \, d \nu + h_{\nu} \geq - r\}.\end{equation}
	
(We note that, a priori, the set $\mathcal K_r$ of $f$-invariant probabilities may be empty.)

For any integer $n \geq 1$  and for all $x \in M$ recall the Definition \ref{DefinitionStatisticalBasin} of the empirical probability $\sigma_n(x)$), and of the p$\omega$-limit set $p\omega(x)$ in the set  $  {\mathcal M}$ of Borel probabilities. We also recall  the   weak$^*$ metric $\dist^*$ in the space   ${\mathcal M}$  of  probability measures constructed by equality (\ref{EqualityDef weak^*distance}).

\begin{lemma}
	\label{Lemma2}  Let $f$ be a $C^1$ expanding map on
	$M$. Let $m$ be the Lebesgue measure on $M$. Fix $r >0$ and let ${\mathcal K}_r  $ be defined by Equality \em (\ref{equationKr}). \em
		Then, for all $0 < \epsilon < r/2$, and for all $\mu \in {\mathcal M}_f$ such that $\mu \not \in {\mathcal K}_r$, there exists $n_0 \geq 1$ and $0<\epsilon ^* \leq \epsilon/3$ such that  \em \begin{equation} \label{equationLemma3}
		m (\{x \in M: \mbox{dist} ( \sigma_{n}(x)  , \mu ) < \epsilon^*\} ) <  e^{ n(\epsilon  - r)} <  e^{- {nr}/{2}}  \ \ \ \ \ \forall \ n \geq n_0.\end{equation}
	
\end{lemma}

\begin{proof} 
	As in the proof of Lemma \ref{Lemma1}, for simplicity along this proof we write ${\mathcal P}^{n }$ instead of $ {\mathcal P}_f^{n } $, and $ h_{\mu} $ instead of $h_{\mu} (f) $.

From Proposition \ref{PropositionExpandingIsExpansive}, $f$ is expansive in the future.
	Let $\alpha >0$ be a expansivity constant for $f$. For the given value of $\epsilon >0$, fix  a uniform continuity modulus $0 <\delta < \alpha$ for $\epsilon/3$ of the function $\psi = - \log |\mbox{det}(Df)|$. Namely 
	\begin{equation}
		\label{Equation2}
		 |\psi(x) - \psi (y)| < \epsilon/3 \mbox { if } \dist (x,y) < \delta.
	\end{equation}
	 For such a value of $\delta$, for the    given measure $\mu \in {\mathcal M}_f$, and for $\epsilon /3$ instead of $\epsilon$, apply Lemma \ref{Lemma1} to  construct  the partition ${\mathcal P}$ in $M$, and the numbers $\epsilon^* >0$ and $n_0 \geq 1$, such that assertions (i), (ii) and (iii) hold. In particular assertion (iii) states that  for any sequence of probability measures $\nu_n$,
	if    $\mu_n:= \frac{1}{n} \sum_ {j= 0}^{n-1} (f^j)^* \nu_n,$  satisfies 
	$\mbox{dist} (\mu_n, \mu) < \epsilon^* \ \ \forall \ n \geq 1$,
	then
	\begin{equation} \label{eqn25} \frac{1}{n } H({\mathcal P}^{n }, \nu_{n }) \leq h_{\mu} + \frac{\epsilon}{3} \ \ \ \forall \ n \geq n_0.\end{equation}
	It is not restrictive to assume that $$\epsilon^* \leq \epsilon/3.$$
	Denote, for all $n \geq 1$:
	\begin{equation}
		\label{Equation6}
			C_{n } := \{x \in M: \ \mbox{dist}(\sigma_{n}(x), \mu) < \epsilon^*\}. 
	\end{equation}

	To prove this Lemma we must prove that
	\begin{equation}
		\label{equationToBeProved}
		m(C_{n}) \leq e ^{n( \epsilon - r )} \ \ \forall \ \ n \geq n_0 \ \ \ \mbox{ (to be proved)} \end{equation}

	Since $f$ is $C^1$ expanding, its derivative $Df_x$ is invertible for all $x \in M$. Thus, by the local inverse map theorem, $f$ is a local diffeomorphism. The compactness of $M$ implies that there exists a uniform value  $\delta_1 > 0 $ such that $f$ restricted to any ball of radius $\delta_1$ is a diffeomorphism onto its image.  Therefore, if the diameter of the partition ${\mathcal P}$ is chosen small enough, the restricted map  $f^n|_{X} : X \mapsto f^n(X)$ is a diffeomorphism for all  $X \in {\mathcal P}^n$ and for all $n \geq 1$. Thus, recalling that $\psi= - \log |\mbox{det}Df|$, we deduce the following equality for all $X \in {\mathcal P}^n$:
	$$m(X \cap C_{n}) = \int _{f^n(X \cap C_{n})}  |\mbox{det}Df^{-n}|  \, d m = \int _{f^n(X \cap C_{n})} e ^{\sum_{j= 0}^{n-1} \displaystyle{\psi \circ f^j }} \, d m.  $$
	Therefore
	\begin{equation}
		\label{Equation1}
		m(C_{n})  = \sum_{X \in {\mathcal P}^n} \int _{f^n(X \cap C_{n})} e ^{\sum_{j= 0}^{n-1}\displaystyle{ \psi \circ f^j } }\, d m.
	\end{equation}
	
	Either $C_{n} = \emptyset$, and Assertion (\ref{equationToBeProved}) becomes trivially proved, or the finite family  of pieces $  \{X \in {\mathcal P}^n : \ X \cap C_{n} \neq \emptyset \} = \{X_1, \ldots, X_N\} $ has $N = N(n) \geq 1$ pieces. In this latter case, choose a single point $y_k \in X_k  \cap C_{n}$ for each $k= 1, \ldots, N$. Denote by
	$Y(n) = \{y_1, \ldots, y_{N}\}   $  the collection of such points. Due to the construction of $\delta >0$ according to Equation (\ref{Equation2}), and since the partition ${\mathcal P}$ has diameter smaller than  $\delta$ (because it satifies (i) of Lemma \ref{Lemma1}), we deduce:
	$$\sum_{j= 0}^{n-1} \psi(f^j(y)) \leq \sum_{j= 0}^{n-1} ( \psi (f^j(y_k)) +   \epsilon/3)   \ \ \forall \ y, y_k \in X_k, \ \forall \  k= 1, \ldots, N.$$

	%\vspace{-.3cm}
	
	\noindent Therefore, substituting in Equality (\ref{Equation1}), 
$$m(C_{n})  \leq e^{n \epsilon/3} \sum_{k= 1}^N e ^{\sum_{j= 0}^{n-1} \displaystyle{\psi (f^j(y_k))}} \, m(f^n(X_k \cap C_{n})).$$
Thus
	$$m(C_{n}) \leq e^{n \epsilon/3} \sum_{k= 1}^N e ^{\sum_{j= 0}^{n-1} \displaystyle{\psi (f^j(y_k))}}.$$
	\noindent Define
	\begin{equation}
		\label{eqn24}
		L:= \sum_{k= 1}^N e ^{\sum_{j= 0}^{n-1} \displaystyle{\psi (f^j(y_k))}}, \ \ \ \ \ \ \ \ \ \ \lambda_k : = \frac{1}{L} \, e ^{\sum_{j= 0}^{n-1} \displaystyle{\psi (f^j(y_k))}} \in (0,1). \end{equation}  Then, $$ \sum_{k= 1}^N \lambda_k = 1$$ and
	\begin{equation}
	\label{Equation3a}
		m(C_{n}) \leq e^{(n \epsilon/3) + \log L},
	\end{equation} 
where
\begin{equation}
	\label{Equation3b}
\log L = \left (\sum_{k= 1}^N \lambda_k \sum_{j= 0}^{n-1} \displaystyle{\psi (f^j(y_k))} \right ) - \left (\sum_{k= 1}^{N} \lambda_k \log \lambda_k \right ).	
\end{equation} 	
(To prove the equality above,  take $\log $ in the equality at right in (\ref{eqn24}), multiply by $\lambda_k$ and take the sum for $k= 1, \ldots, N.$ )

	 Define the probability measures
	 \begin{equation}
	 	\label{Equation7a}
	 		\nu_{n} := \sum_{k= 1}^N \lambda_k \delta_{y_k}, 
	 \end{equation}
  \begin{equation}
 	\label{Equation7b}
 	 \mu_{n } := \frac{1}{n} \sum _{j= 0}^{n-1} (f^j)^* (\nu_n) = \sum_{k=1}^N \lambda _k \frac{1}{n} \sum_{j= 0}^{n-1} \delta_{f^j(y_k)} = \sum_{k= 1}^N \lambda_k\sigma_{n}(y_k).  
 \end{equation}
	(To prove the above equality at right recall Definition \ref{DefinitionStatisticalBasin} of the empirical probability measures $\sigma_n(y_k)$.)

Then,	
\begin{equation}
	\label{Equation4}
	\sum_{k= 1}^N \lambda_k \sum_{j= 0}^{n-1} \displaystyle{\psi (f^j(y_k))} = n\int \psi \, d \mu_n.
\end{equation}
Recall that for any piece $X_k \in \mathcal P ^n $ such that $C_n \cap X_k \neq \emptyset	$ 
we have chosen a single point $y_k \in C_n \cap X_k$. Then $\nu_n(X_k) = \lambda_k \delta_{y_k} (X_k) = \lambda_k$, and we deduce that
\begin{equation}
	\label{Equation5}
-	\sum_{k= 1}^{N} \lambda_k \log \lambda_k  = H({\mathcal P}^n, \nu_n).
\end{equation}

Therefore, combining Equations (\ref{Equation3a}),  (\ref{Equation3b}), (\ref{Equation4}) and (\ref{Equation5}), we obtain 
\begin{equation}
	\label{Equation80}
		m(C_{n}) \leq \mbox{exp}\Big({\frac{n \epsilon}{3} + \log L}\Big) = \mbox{exp} \Big ( n \Big ( \frac{\epsilon}{3} + \int \psi \, d \mu_n + \frac{H({\mathcal P}^n, \nu_n)}{n} \Big ) \Big ).
\end{equation}

	Now, we assert that \begin{equation} \label{eqntobeproved}\dist (\mu_n, \mu) < \epsilon^* \leq \frac{\epsilon}{3} \ \ \ \forall \ n \geq 1.\end{equation}
	In fact, by construction $y_k \in C_n$ for all $k= 1, \ldots, N$, Thus, from Equality (\ref{Equation6}), we have $$\mbox{dist}(\sigma_n (y_k), \mu) < \epsilon^*.$$
Recalling Lemma \ref{LemmaBolasConvexas},  the ball in ${\mathcal M}$ of center $\mu$ and radius $\epsilon^*$ is convex. Thus, any convex combination of the measures $\sigma_n (y_k)$ belongs to that ball. From Equality (\ref{Equation7b}) at right, $\mu_n $ is a convex combination of the measures  $\sigma_n(y_k) $. We deduce that $\mu_n$  belongs to that ball. Hence, inequality (\ref{eqntobeproved}) is proved.  So equation (\ref{eqn25}) holds.

	Combining Equations   (\ref{eqn25}) and (\ref{Equation80}), we deduce that
	$$ m(C_{n }) \leq \mbox{exp} \Big ( n \Big ( \frac{2 \cdot \epsilon}{3} + \int \psi \, d \mu_n  + h _{\mu} \Big ) \Big ) \ \ \forall \ n \geq n_0. $$
	Besides, from inequality (\ref{eqntobeproved}) and the construction of the weak$^*$-metric $\dist^*$ in ${\mathcal M}$  with $\phi_0 = \psi$ (recall Equality (\ref{EqualityDef weak^*distance})), we deduce that
	$$\left|\int \psi \, d \mu_n -  \int \psi \, d \mu \right| < \frac{\epsilon}{3}, \ \ \ \int \psi \, d \mu_n <  \int \psi \, d \mu  + \frac{\epsilon}{3}.$$

	Therefore, we obtain
	\begin{equation} \label{equationPrincipal} m(C_{n }) \leq \mbox{exp} \Big ( n \Big (   \epsilon  + \int \psi \, d \mu  + h _{\mu} \Big ) \Big ) \ \ \forall \ n \geq n_0. \end{equation}
	Finally, by hypothesis $\mu \not \in {\mathcal K}_r$. Thus,
	$ \int \psi \, d \mu  + h _{\mu} < - r.$
	Substituting this latter inequality in (\ref{equationPrincipal}), we
	conclude (\ref{equationToBeProved}), ending the proof. \end{proof}

The following lemma is a well-known elementary result in Probability Theory. We will apply it in the particular  case for which  $M $ is a compact Riemannian manifold and $\mathcal B$ is the Borel $\sigma$-algebra of subsets of $M$.
\begin{lemma} {\bf (Borel-Cantelli)}
	\label{LemmaBorel-Cantelli}
	Let $\mu$ be a probability measure on a measurable space $(M, \mathcal B)$. Let $\{C_n\}_{n\geq 1}$ be a sequence of measurable subsets $C_n \subset M$ such that
	$$  \sum_{n=1} ^{+\infty} \mu (C_n) < + \infty.  $$
	Then
	$$  \mu \left( \bigcap_{N \geq 1} \bigcup _{n \geq N} C_n \right) = 0.  $$
\end{lemma}
\begin{proof}
	The sequence $\left\{ \bigcup_{n \geq N}C_n\right\}_{N \geq 1}$ is (not necessarily strictly) decreasing with $N$.
	Then
	$$  \mu \left( \bigcap_{N \geq 1} \bigcup _{n \geq N} C_n \right) = \lim_{N \rightarrow +\infty}
	 \mu \left(\bigcup _{n \geq N} C_n \right) \leq  \lim_{N \rightarrow +\infty}
	\sum_{n=N}^{+ \infty} \mu(C_n).$$
	Finally,    $\lim_{N \rightarrow +\infty}
	\sum_{n=N}^{+ \infty} \mu(C_n) = 0$ because
	$\sum_{n=N}^{+ \infty} \mu(C_n) $ is the tail of the convergent series  $  \sum_{n=1}^{+ \infty} \mu(C_n) $.
\end{proof}

\subsection{End of the proof of Theorem \ref{MainTheorem}}

\begin{proof}
	
We will prove that any pseudo-physical measure $\mu$ satisfies Pesin's Entropy Formula, namely,
 for the $C^1$- expanding map $f$, according to Proposition \ref{PropositionPesinForm=EqState}:
	$$h_{\mu} (f) + \int \psi\, d \mu = 0    \ \ \ \ \ \ \ ,$$
	where $$ \psi := -\log |\mbox{det} Df |.$$	
	For any $r >0$ consider the   compact set
	${\mathcal K}_r \subset {\mathcal M}$ defined by Equality (\ref{equationKr}).
	Since $\{{\mathcal K}_r\}_r$ is decreasing when decreasing  $r $, we have
	$${\mathcal K}_0  = \bigcap_{r >0}{\mathcal K}_r, \ \ \mbox{ where} $$
	$${\mathcal K}_0 := \Big\{ \mu \in {\mathcal M}_f: \ \ \int \psi \, d \mu  + h_{\mu}(f) \geq 0\Big\}.$$	
	By Margulis-Ruelle's inequality (see Theorem \ref{TheoremRuelleIneq}) and Corollary \ref{CorollaryLiouvilleFormulaForExpandingMaps} applied to $C^1$ expanding maps, we have
	\begin{equation} \label{eqn29} h_{\mu} (f) \leq \int \log |\mbox{det} Df| \, d \mu = - \int \psi \, d \mu \ \ \forall \ \mu \in {\mathcal M}_f. \end{equation}
	Therefore, the (a-priori maybe empty) set ${\mathcal K}_0$ is composed by all the invariant measures $\mu$  such that
	$$\int \psi \, d \mu  + h_{\mu}(f) = 0,$$ or, in other words, ${\mathcal K}_0$ is the set of invariant measures $\mu$
	that satisfy Pesin's Entropy Formula.
	So, to prove that any pseudo-physical measure $\mu$ satisfies Pesin's Entropy Formula,  we must  prove that $ \mu \in  {\mathcal K}_r$ for all $r >0$.
	
	Assume by contradiction that there exists $r >0$ such that the pseudo-physical measure $\mu $ does not belong to ${\mathcal K}_r$. From Lemma \ref{Lemma2}, there exists $n_0 \geq 1 $ and $\epsilon^* >0$
	such that,   \begin{equation} \label{eqn28} m \big(\{x \in M: \mbox{dist}^*(\sigma_{n}(x), \mu) <\epsilon^*) \}\leq
	e^{-nr/2} \ \ \ \forall \ n \geq n_0,\end{equation} where $m$ denotes the Lebesgue measure.

	From Definition \ref{DefinitionPseudoPhysical} of pseudo-physical measure,  for any $\epsilon^* > 0$  the set
	$$A = \{ x \in M: \ \ \dist(p\omega (x), \mu) < \epsilon^*\}$$ has positive Lebesgue measure: $m(A) >0$. For each $n \geq 1$, denote
	$$C_n := \{x \in M: \ \ \dist (\sigma_n(x), \mu) < \epsilon^*\}.$$ Apply Definition \ref{DefinitionStatisticalBasin} of the set $p \omega (x)$ in $\mathcal M$ composed by  the weak$^*$-limits of all the convergent subsequences of $\{\sigma_n(x)\}$. Therefore,
	$$A \subset \bigcap_{N \geq 1} \bigcup_{n \geq N} C_n,$$
	So, we deduce the following inequality:
\begin{equation} \label{Equation99}
	m \Big(\bigcap_{N \geq 1} \bigcup_{n \geq N} C_n  \Big) \geq m(A) >0.  
\end{equation}
	
	But inequality (\ref{eqn28}) implies that $\sum_{n=1}^\infty m (C_n) < +\infty$; hence, applying   Borel-Cantelli Lemma (see Lemma \ref{LemmaBorel-Cantelli}), it follows that
	$$ m\Big(\bigcap_{N \geq 1} \bigcup_{n \geq N}C_n\Big)=0,$$  contradicting inequality (\ref{Equation99}).
	 \end{proof}

\section{Examples} \label{Section5-Examples}

\begin{definition} \label{DefinitionC1+alpha}
	 {\bf ($C ^{1+ \alpha}$-maps.)} \em  We say that a map $f:M\mapsto M$  is  $C^1$ plus Hölder, and denote $f \in  C ^{1+ \alpha}$,  if $f$ is differentiable, with continuous derivative and besides there exists  $  \alpha>0$ such that the derivative of $f$ is $\alpha$-Hölder continuous. Namely, for some constant $K>0$, the following inequality holds:
	 $$\| Df_x -Df_y\| \leq K (\mbox{dist} (x,y) )^{\alpha} \ \mbox{ for all } x,y \in M. $$  
	
\end{definition}

As said in the introduction, the theory of existence of physical measures  (see Definition \ref{DefinitionPhysicalMeasure}) for   $ C^{1+\alpha}$ expanding maps  (see Definition \ref{DefinitionExpandingMaps}) is well known. So, we will look only for examples $f$ that are $C^1$ but not $C^{1+\alpha}$ for any $\alpha >0$.

Precisely, we will construct examples of $C^1\setminus C^{1+\alpha}$ expanding maps on the circle $S^1$ and on the 2-torus $T^2$ and study their pseudophysical measures (see Definition \ref{DefinitionPseudoPhysical}). In the examples that we will  construct along this section, the pseudophysical measures will be indeed physical. Applying Theorem \ref{MainTheorem} we know that all these measures will satisfy Pesin's Entropy Formula (see Definition \ref{DefinitionPesinFormula}).

In all the examples that we present along this section the invariant physical measures are absolutely continuous with respect to the Lebesgue measure. Thus, these examples do not belong to the $C^1$ generic family of maps found in \cite{avilabochi}.
 
First, we will construct two examples in $S^1$, and second,  an example on $T^2= S^1 \times S^1$ that will be the product of the examples previously constructed on $S^1$. The two examples on $S^1$ are taken from Subsection 2.1 of \cite{Araujo-Santos-2019}.

As the circle $S^1$ is the result of identifying the extremes of a closed interval, we will construct the examples by constructing maps on the interval having finitely many $C^1$- expanding, order preserving continuity pieces, that 
are surjective on each continuity piece, and  the maps and their derivatives glue well in the extremes of each continuity piece so they can be seen as $C^1$ maps on $S^1$. Precisely, in the example of Figure \ref{Figure1}, we first take a real number $0<b_0<1$ and construct any piecewise $C^1$ map  $f: [-1,1] \to [-1,1]|$ that satisties the following properties, so it defines a $C^1$-expanding and order preserving map on the circle $S^1 = [-1 ,1]/_{(-1\sim 1)}$:
$$f(-1)= f(-b_0^+) = f (b_0^+) = -1, \ \ \ f(1)= f(-b_0^-)= f(b_0^-)= 1  $$
(where $f(x_0^+)$ denotes the lateral limit of $f(x)$ when $x\rightarrow x_0^+$ by the right, and $f(x_0^-)$ the lateral limit by the left), 
$$  f' (x) > 1 \ \ \forall \ x \in (-1,-b_0) \cup (-b_0, b_0) \cup (b_0,1) $$ and   the following limits exist and are finite:
$$f'(-1^+)= f'(1^-),    \ \ f'(-b_0^+)= f'(-b_0^-) ,\ \   f'(b_0^+)= f'(b_0^-). $$

 \begin{figure}[ht]
 	\centering
 	\hspace{-0.2in}\vspace{0.2in}\includegraphics[scale=.40]{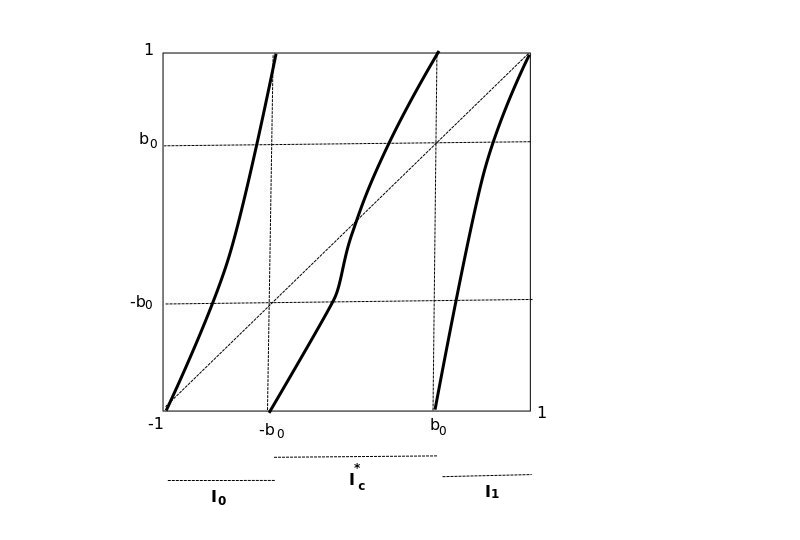}
 	\caption{An expanding map $f$ on the circle $S^1$ with index 3.  \label{Figure1}}
 \end{figure}

\subsection{Dynamically Defined Cantor Sets} \label{SubsectionCantorSetsDynDefined}

Before giving the examples, we will recall the construction of a Cantor set that is dynamically defined in an interval by an expanding map defined in two closed disjoint subintervals. 

To fix the ideas we will use the example  $f: [-1,1] \to [-1,1]$ described above (Figure \ref{Figure1}). 

Denote $I_0:=[-1, -b_0], \  I_1 :=[b_0,1]$ and $I^*_c := (-b_0, b_0)$,  $G := f|_{I_0 \cup I_1}$,  $G_0: =f |_{I_0 }$ and $G_1:= f |_{I_1}$ (Figure \ref{Figure2}).
Construct the following compact set
\begin{equation} \label{equation00}
	K:= \bigcap_{n=1} ^{\infty} G^{-n} ([-1,1]) \subset I_0 \cup I_1
\end{equation}

Each point $x \in K$ has an \em itinerary, \em  which is defined as the following sequence of  0's and 1's:
$$ \underline{a} = \underline{a} (x) :=  a_0, a_1, \ldots, a_n, \ldots \in 2^{\mathbb N}\ \mbox{where }    G^n (x) \in I_{a_n} \ 
\mbox { for all }n \geq 0.
$$
For fixed $n \geq  1$, denote by $\underline{a} _n$ the following word of lenght $n$ composed by 0's and 1's:
$$ \underline{a} _n :=  a_0, a_1, \ldots, a_{n-1}  \in \{0,1\}^n.
$$
Also denote  \begin{equation}
	\label{EquationConstruccionAtom}
	I_{\underline{a}_n}:= G_{a_0}^{-1} \circ G_{a_1}^{-1} \circ \ldots \circ G_{a_{n-1}}^{-1} ([-1,1])
\end{equation}

Note that, for each fixed word  $\underline{a}_n$, the set     $I_{\underline{a}_n}$  is a closed  interval because $G_0$ and $G_1$ are strictly increasing continuous maps. By construction, each interval $ I_{\underline{a}_n} $
is composed by all the points of $I_0 \cup I_1$ that share  the same finite word of the itinerary, from time 0 to time $n-1$.

From the above construction, for each fixed $x \in K$,  being $\underline{a} _n (x)$ the finite word of length $n\geq 1$ taken from the fixed sequence $\underline{a} (x)$, we have  
$ x \in \bigcap _ { n=1} ^{\infty}  I_{\underline{a}_n(x)}.$
But, since $G_0$ and $G_1$ are $C_1$-expanding, the length $|I_{\underline a _{n+1} (x)}|$ 
 of the interval $I_{\underline a _{n+1} (x)}$ satisfies 
$$ |I_{\underline a _{n+1} (x)}|   \leq \lambda |I_{\underline a _{n} (x)}| \ \mbox{for all } n \geq 1, \ \mbox{with } 0 <\lambda<1,$$
where $ \lambda =( \min_{u \in I_0 \cup I_1} G ' (u))^{-1} .$ We conclude that  $|I_{\underline a _{n} (x)}|  \leq  \lambda ^n \rightarrow 0$, with $n \rightarrow +\infty $. So:
$$\{x\} =\bigcap _ { n=1} ^{\infty}  I_{\underline{a}_n(x)}.    $$

Taking now all the finite words $\underline a _n \in \{0,1\}^n $, we obtain for each time $n \geq 1$ the following equality 
$$G^{-n} ([-1,1]) = \bigcup_{\underline a _n \in \{0,1\}^n} I_{\underline a _n },   $$
and together with Equality (\ref{equation00}):

\begin{eqnarray} \label{equationCantorSet}
	K= \bigcap _{n=1}^{\infty} \bigcup_{\underline a _n \in \{0,1\}^n} I_{\underline a _n } 	
\end{eqnarray}
Taking into account  that $|I_{\underline a _{n} (x)}|  \rightarrow 0$, with $n \rightarrow +\infty $, we deduce that $K$ is a Cantor set. 

\begin{definition} \label{DefinitionCantorSetDynDefined} \em {\bf (Dynamically defined Cantor set.)}
We say that $K$ is  the \em dynamically defined Cantor set   by the map $G$ \em if it satisfies Equality  (\ref{equation00})   (and hence also Equality (\ref{equationCantorSet})).
\end{definition}

\begin{definition} {\bf{(Atoms generating the Cantor set.)}}
	\label{DefinitionAtom} \em
For fixed $n\geq 1$, we call the $2^n$ pairwise disjoint closed intervals $	I_{\underline{a}_n}$, $\underline{a}_n \in \{0,1\}^n$ defined by Equality (\ref{EquationConstruccionAtom}) and satisfying Equality (\ref{equationCantorSet}), \em the atoms of generation $n$ \em  generating the Cantor set $K$.
\end{definition}
We assert that, for fixed $\underline{a}_n \in \{0,1\}^n$,  the atom  $I_{\underline{a}_n}$ of generation $n$ contains exactly two atoms $I_{a_0 a_1 \ldots a_{n-1} 0}  $ and $I_{a_0 a_1 \ldots a_{n-1}1} $ of generation $n+1$ that are obtained from $I_{\underline{a}_n}$, at left and at right respectively, after removing an open interval  $I^*_{\underline{a}_n}$ (see Figure \ref{Figure3}). This is the classical construction of a Cantor set in the interval, even if it is not dynamically defined. 

In fact, for $n=1$, note that $G_0$ and $G_1$ are continuous, strictly increasing and surjective on $[-1,1] = I_0 \cup I^*_c \cup I_1$ (see Figure \ref{Figure2}). Applying Equality (\ref{EquationConstruccionAtom}), the two atoms of generation 2 inside $I_{0}$ are $I_{0 0} $ and $I_{01 }$,   the preimages by $G_0$  of $I_0$ and $I_1$ respectively (see Figure \ref{Figure2}).
They are two closed intervals obtained from $I_0$, at left and right respectively,  after removing the open interval
\begin{equation}
	\label{equalityGapI*0}
	I^*_0 = G_0 ^{-1} (I^*_c).
\end{equation} Namely  $I_0 = I_{00} \cup I_0^* \cup I_{01}$ Analogously,  $I_1 = I_{10} \cup I_1^* \cup I_{11}$ where  \begin{equation}
\label{equalityGapI*1}
I^*_1= G_1 ^{-1} (I^*_c).
\end{equation}

 \begin{figure}[ht]
	\centering
	\hspace{-0.2in}\vspace{0.2in}\includegraphics[scale=.40]{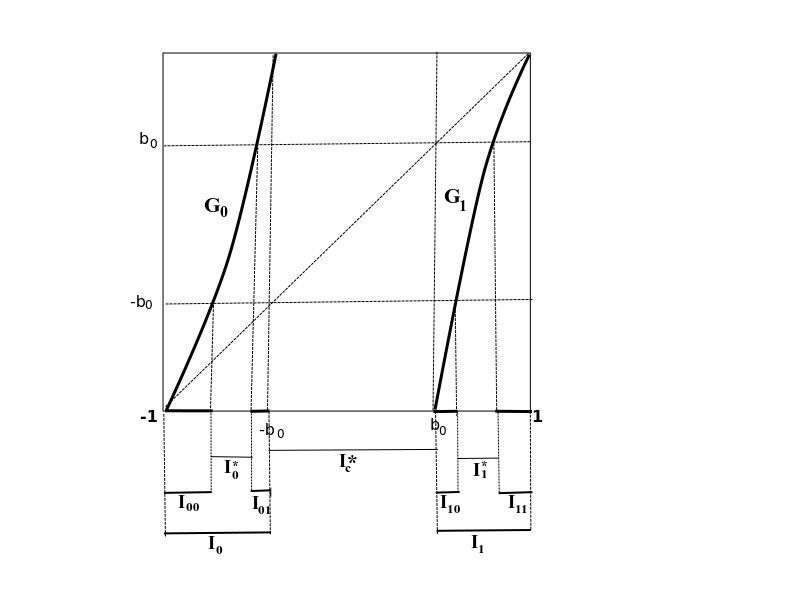}
	\caption{The expanding map $G$ defined in the intervals $I_0$ e and $I_1$, where $G_0 = G|_{I_0}, \ G_1 = G|_{I_1}$.  The atoms of generation 1 of the Cantor set dynamically defined by $G$ are $I_0= G_0^{-1}([-1,1])$ and $I_1 = G_1^{-1}([-1,1])$. Inside them there are  open intervals, the gaps of generation 1: $I_0^*= G_0^{-1}(-b_0, b_0)$ e $I_1^*=G_1^{-1}(-b_0, b_0)$, respectively. Removing these gaps from the atoms of generation 1, we obtain the atoms of generation 2: $I_ {00} = G_0^{-1}(I_0), \ I_{01} = G_0^{-1}(I_1),  I_ {10 }= G_1^{-1}(I_0), \ I_{11} = G_1^{-1}(I_1)    $. \label{Figure2}}
\end{figure}

By induction on $n\geq 2$, the continuity and strictly increasing property of the maps $G_0$ and $G_1$ ensure that there are exactly 2 atoms of generation $n+1$  inside  $I_{a_0 a_1  \ldots  a_{n-1}}$, obtained by taking the preimage by $G_{a_0}$ of the closed interval $$I_{a_1 \ldots   a_{n-1}} =  I_{a_1 \ldots   a_{n-1} 0} \cup  I^*_{a_1 \ldots   a_{n-1} } \cup  I_{a_1  \ldots  a_{n-1} 1}.$$ 
These two atoms of generation $n+1$ are the preimages by $G_{a_0}$ of the closed subintervals $I_{a_1 \ldots   a_{n-1} 0}$ and $I_{a_1  \ldots  a_{n-1} 1}$ respectively. 
They are obtained after removing from $ I_{a_0 a_1  \ldots a_{n-1} }$  the open interval
\begin{equation}
	\label{equalityGapI*a_n} 
	I^*_{a_0 a_1  \ldots a_{n-1} } := G_{a_0}^{-1} (I^*_{a_1 \ldots    a_{n-1} }),
\end{equation}
at left and right  respectively (see Figure \ref{Figure3}). Namely,
\begin{equation}
	\label{EqualityAtomDividedByGap} 
	I_{a_0, a_1 \ldots   a_{n-1}} =  I_{a_0 a_1 \ldots   a_{n-1} 0} \cup  I^*_{a_0 a_1 \ldots   a_{n-1} } \cup  I_{a_0 a_1  \ldots  a_{n-1} 1}.
\end{equation}

\begin{definition} {\bf{(Gaps of the Cantor set.)}}
	\label{DefinitionGaps}
	\em
	
We call  the open interval $I_c^*$ (see Figure \ref{Figure2})  the \em gap of generation 0 \em of the Cantor set $K$.	For fixed $n\geq 1$, we call the $2^n$ pairwise disjoint open intervals  $	I^*_{\underline{a}_n}$, contained respectively in the atoms $	I_{\underline{a}_n}$ as defined above for  $\underline{a}_n \in \{0,1\}^n$ by Equalities (\ref{equalityGapI*0}),  (\ref{equalityGapI*1}) and (\ref{equalityGapI*a_n}),   \em the gaps of generation $n$ \em  of the Cantor set $K$.
	
Observe that the countably infinite family of all the atoms of all generation is a family of pairdisjoint open intervals.	
\end{definition}

From the construction of the atoms and gaps of $K$ we deduce that the union of the atoms of generation $n+1$ is obtained from the union of the atoms of generation $n$ by removing all the gaps of generation $n$.   
So, using Equality (\ref{equationCantorSet}), we obtain that
\begin{equation}
	\label{EqualityCantorSetComplementoDeGaps}
	K= [-1,1] \setminus \bigcup_{n=0} ^ {\infty } \ \  \bigcup_{\underline {a}_n \in \{ 0,1\}^n} I^*_{\underline {a}_n},
\end{equation}
with the convention  $\{ 0,1\}^0 = \{c\}$. In other words, the Cantor set $K$ is obtained from the interval by removing all the countably infinite many pairwise disjoint gaps.

\subsection{Bowen's construction of a dynamical defined Cantor set with positive Lebesgue measure.} \label{SubsectionBowen}

The examples of $C^1\setminus C^{1+ \alpha}$-expanding maps on $S^1$ and on $T^2$ that we will construct in the following subsections will be based on Bowen's method to construct dynamical defined Cantor sets in the interval with positive Lebesgue measure. Along this subsection we follow Bowen's construction in   \cite{Bowen}.

The construction is done in two steps: First, we will construct a Cantor set $K$ in the interval $[-1,1]$ contained in the two disjoint  closed subintervals $I_0$ and $I_1$. This $K$ will not be dynamical defined yet, but it will have  positive Lebesgue measure.  Second, we will construct the map $G$  as in Figure \ref{Figure2} such that $G_0 := G|_{I_0} $ and  $G_1 := G|_{I_1} $  are $C_1\setminus C^{1+\alpha}$ expanding and the Cantor set constructed in the first step becomes dynamically defined by $G$ according to Definition \ref{DefinitionCantorSetDynDefined}.

\noindent {\bf  Step 1: Construction of the Cantor set $K$ with positive Lebesgue measure.}  We use the notation of Subsection \ref{SubsectionCantorSetsDynDefined}. 

Take $0 <b_0<1$ and define the gap of generation 0 to be $I^*_c := (-b_0, b_0) \subset [-1,1]$, and the two atoms of generation 1: $I_0 :=[-1, -b_0]$ and $I_1 := [b_0, 1].$

Take a sequence of real numbers $\alpha_n>0$ for all $n \geq 1$ such that
\begin{equation}
	\label{Equation200}
\sum_{n=1}^{\infty} \alpha_n < 2(1-b_0), \ \ \alpha_{n+1} < \alpha_n \mbox{ for all } n \geq 1, \ \lim_{n\rightarrow + \infty} \frac{\alpha_{n+1}}{\alpha_n} = 1, \ \ \alpha_1 <  2b_0 	
\end{equation} 
(for instance $\alpha_n := (k n^2)^{-1}$ for some constant $k >0$ large enough).

Let us construct the 2 gaps $I^*_0$ and $I^*_1$ of generation 1 and the  4 atoms $I_{00}, I_{0,1}, I_{10}$ and $I_{11}$ of generation 2:
For $a_0 \in \{0,1\} $ construct the open interval $I^*_{a_0} \subset I_{a_0}$ centered at the centre of     
$I_{a_0} $ and with length $\alpha_1/2$.  The atoms $I_{a_0 0}$ and $I_{a_0 1}$ of generation 2 are the closed intervals obtained from $I_{a_0} $ after removing the gap $I^* _{a_0}$ of generation 1, at left and right of $I^* _{a_0}$ respectively.

By induction on $n\geq 1$,  for all $\underline {a}_n \in \{0,1\}^n$ let us construct the gaps $I^*_{\underline{a}_n}$ of generation $n$ and the atoms 
$I_{\underline{a}_n 0}$ and $I_{\underline{a}_n 1}$ of generation $n+1$:
$I^*_{\underline{a}_n}$ is the open interval contained in the atom $I_{\underline{a}_n }$ of generation $n$,   centered at the centre of     
$I_{\underline{a}_n }$ and with length $\alpha_n/2^n$.  The atoms $I_{\underline{a}_n 0}$ and $I_{\underline{a}_n 1}$ of generation $n+1$ are the closed intervals obtained from $I_{\underline{a}_n }$ after removing the gap $I^*_{\underline{a}_n}$, at left and right of $I^*_{\underline{a}_n}$ respectively (see Figure \ref{Figure3}).

\begin{figure}[ht]
	\centering
	\hspace{-0.2in}\vspace{0.2in}\includegraphics[scale=.40]{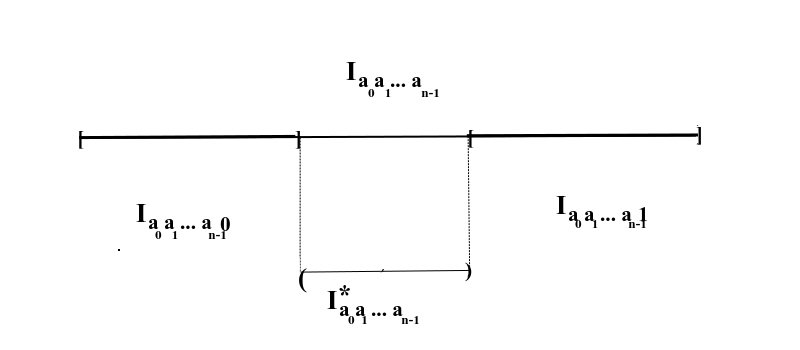}
	\caption{The atom of generation $n$ named $I_{\underline a_n}$ where $\underline a _n = a_0 a_1 \ldots. a_{n-1} \in \{0,1\}^n$. Inside it, there is a gap of generation $n$ named $I^*_{\underline a_n}$. Removing this gap from the atom  $I_{\underline a_n}$,  the two atoms of generation $n+1$ are obtained: $ I_{\underline a_n 0}  $  at left of the gap, and $ I_{\underline a_n 1}   $ at right. \label{Figure3}}
\end{figure}

We construct in what follows the Cantor set $K$ defined by Equality (\ref{equationCantorSet}), or equivalently by Equality (\ref{EqualityCantorSetComplementoDeGaps}).
From Equality (\ref{EqualityCantorSetComplementoDeGaps}) the Cantor set  $K$ is the complement in the interval $[-1,1]$ of the union of all the gaps. Then, 
$$ \mbox{Leb} (K) = \mbox{Leb} ([-1,1]) - \sum_{n=0}^{\infty} \sum_{\underline{a}_n \in \{0,1\}^n} \mbox{Leb} (I^*_{\underline{a}_n})$$ 
(recall the notation convention $\{0,1\}^0 := \{c\}$ ).
The gap  $I^*_c$ of generation 0 has lenght $2b_0$ and for all $n \geq 1$, the $2^n$ gaps of generation $n$ have all the same length $\alpha_n/2^n$. Then,
$$ \mbox{Leb} (K) = 2 -  2b_0 - \sum_{n=1}^{\infty} 2^n \frac{\alpha_n}{2^n} = 2(1-b_0) - \sum_{n=1}^{\infty} \alpha_n >0 $$
(recall the inequality at the left in (\ref{Equation200})).

\noindent {\bf Step 2. Construction of the map $G$ that makes $K$ dynamically defined by $G$.}

We need  that the map $G: I_0 \cup I_1 \to [-1,1]$ to be constructed satisfied Equalities (\ref{equalityGapI*0}), (\ref{equalityGapI*1}) and (\ref{equalityGapI*a_n}). Then, we will construct $G$ on each gap of $K$ according to those equalities.

Let us start by the gaps of generation 1: For any $a_0 \in \{0,1\}$, construct $G|_{I^*_{a_0} }:  I^*_{a_0} \to I^*_c$ as any $C^1$ expanding, order preserving and surjective map, such that 
\begin{equation} \label{Equation201}
\lim_{x \rightarrow \partial {I^*_{a_0}}^  {\pm}  } G' (x) = 2 	
\end{equation}
(the notation $x \rightarrow \partial {I^*_{a_0}}^  {\pm}$ means
 $x$ going to the right extreme of the interval  $I^*_{a_0}$ by the left, and $x$ going to the left extreme by the right).

Such a construction is possible because $| I^*_c  |/ |I^*_{a_0}     | = 2b_0 / (\alpha_1/2)= 2 \cdot (2b_0)/\alpha_1 > 2 $
(recall inequality at the right in (\ref{Equation200})).

Now, let us construct $G$ in the gaps of generation $n+1$ for all $n\geq 1$: For any $\underline{a}_{n+1} = a_0 a_1 \ldots a_{n-1} a_{n} \in \{0,1\}^{n+1}$ construct $G|_{I^*_{\underline{a}_{n+1}}}:  I^*_{a_0 a_1 \ldots a_{n-1} a_{n}} \mapsto   I^* _  {a_1 \ldots a_{n-1} a_{n}} $ as any $C^1$ expanding, order preserving and surjective map, such that

\begin{equation} \label{Equation202}
\lim_{x \rightarrow \partial {I^*_{\underline{a}_{n+1}}}^  {\pm}  } G' (x) = 2, 	
\end{equation}

Such a construction is possible because $|  I^* _  {a_1 \ldots a_{n-1} a_{n}} |/   |   I^* _  {a_0 a_1 \ldots a_{n-1} a_{n}}      | = (\alpha_n/2^n) / (\alpha_{n+1}/2^{n+1})= 2 \cdot \alpha_n /\alpha_{n+1} > 2 $
(recall the second inequality at the left in (\ref{Equation200})).

Besides, when increasing $n$ from 1 to infinity, construct $G$ on each gap of generation $n+1$ such that , for some sequence $\epsilon_n \rightarrow 0$ with $n \rightarrow + \infty$ the following inequality holds
\begin{equation} \label{Equation203}
 |G'(x)-2| < \epsilon _n  \ \mbox{ for all } x\in {I^*_{\underline{a}_{n+1}}} \ \mbox {for all } \underline{a}_{n+1}  \in \{0,1\}^{n+1}.  	
\end{equation}

Such a condition is possible to obtain because, due to the equality in (\ref{Equation200}) we know that $\lim _{n\rightarrow +\infty} \alpha_{n+1}/\alpha_n =1 $.   So, $$\lim_{n\rightarrow +\infty}\frac{ | I^* _  {a_1 \ldots a_{n-1} a_{n}} |} {   |   I^* _  {a_0 a_1 \ldots a_{n-1} a_{n}}      |} = \lim_{n\rightarrow +\infty} 2 \cdot \frac{\alpha_n }{\alpha_{n+1} }= 2 $$

Once the map $G$ is defined in all the gaps of the Cantor set $K$ contained in $I_0\cup I_1$ according to the conditions (\ref{equalityGapI*0}), (\ref{equalityGapI*1}) and (\ref{equalityGapI*a_n}), we extend it continously to $K$ so we have $G$ defined continously in $I_0\cup I_1$ as in Figure \ref{Figure2}, being $C^1$expanding in the union of all the gaps. Due to Equalities (\ref{Equation201}) and (\ref{Equation202}), and to Inequality (\ref{Equation203}), the map $G$ is $C^1$ also in $K$, being expanding in  both  intervals $I_0$ and $I_1$ and $G'(x) =2$ for all $x \in K$.  

Since $G$ satisfies Equalities  (\ref{equalityGapI*0}), (\ref{equalityGapI*1}) and (\ref{equalityGapI*a_n}), and  the Cantor set $K$ constructed in the Step 1 satisfies Equalities (\ref{equationCantorSet}) and (\ref{EqualityCantorSetComplementoDeGaps}), $K$ also satisfies Equality (\ref{equation00}). So, $K$ is the dynamically defined Cantor set defined by $G$, as stated.

\noindent {\bf Assertion.}  \em The map $G$ constructed above is not $C^{1 + \alpha}$ for any $\alpha >0$. \em  

In fact, since $\mbox{Leb}(K) >0$, this assertion  is a restatement of the following proposition:

\begin{proposition} \label{PropositionLebK=0 if C1+Holder}
	Let $G: I_0 \cup I_1 \mapsto [-1,1]$ be a $C^{1 + \alpha}$-expanding map such that $G_0 = G|_{I_0} $ and $G_1 = G|_{I_1} $ are surjective on $[-1,1]$ as in Figure \ref{Figure2}. Then, the Cantor set dynamically defined by $G$ according to Definition \ref{DefinitionCantorSetDynDefined}, has zero Lebesgue measure.
\end{proposition}
Proposition  \ref{PropositionLebK=0 if C1+Holder} is a well known result of the ergodic theory of $C^{1 + \alpha}$ dynamical systems. 
For the sake of completeness, we include its proof here.

\begin{proof}
	For fixed $\underline{a}_n =a_0a_1 \ldots a_{n-1}  \in \{0,1\}^n, \ n\geq 1$, recall the construction of the atom $I_{\underline{a}_n}$ of the Cantor set $K$ given by equality (\ref{EquationConstruccionAtom}). 
	%Is is composed by all the points of $I_0 \cup I_1$ that shar the same finite word $\underline{a}_n = a_0 a_1 \ldots a_{n-1}$ of the itineraty from time 0 to time $n-1$.
	Since $G$ is $C^1$, let us compute the length (the Lebesgue measure) of the closed interval $I_{\underline{a}_n}$ by applying the Mean Value Theorem of the differential calculus.  We obtain
	$$ |[-1,1]|=(G ^n)'(x_{\underline{a}_n}) \cdot |I_{\underline{a}_n}| $$
	for some point  $x_{\underline{a}_n}\in I_{\underline{a}_n}$.
	
	Now, consider the gap  $I^*_{\underline{a}_n}\subset I_{\underline{a}_n}$  as in Equality (\ref{EqualityAtomDividedByGap}). Using Equalities (\ref{equalityGapI*0}),  (\ref{equalityGapI*1}) and (\ref{equalityGapI*a_n}), we obtain 
	$$ I^*_c = G^n (I^*_{\underline{a}_n})$$
	and, applying again the Mean Value Theorem, there exists some point $y_{\underline{a}_n}\in I^*_{\underline{a}_n} \subset I_{\underline{a}_n}$ such that 
		$$ | I^*_c|=(G ^n)'(y_{\underline{a}_n}) \cdot |I^*_{\underline{a}_n}| $$

We obtain
\begin{equation}
	\label{equation333}
		\frac{|I^*_{\underline{a}_n}|} {|I_{\underline{a}_n}|} =  \frac{(G ^n)'(x_{\underline{a}_n}) } {(G ^n)'(y_{\underline{a}_n}) } \cdot r_0 \ \mbox{where } 0< r_0 := \frac {|I^*_c|}{|[-1,1]|} <1. 
\end{equation}
	Note that the points $x_{\underline{a}_n} $ and $ y_{\underline{a}_n}$ have the same itinerary up to time $n$ because both belong to the same atom $I_{\underline{a}_n} $. Therefore we can apply the following well known lemma that holds for $C^{1+\alpha}$- expanding maps:
	
	\noindent {\bf Bounded Distortion Lemma.}
\em 
If $G$ is a $C^{1+\alpha}$ expanding map as in Figure \ref{Figure2}, then there exists a constant $c>0$ such that for all $n\geq 1$, if $x$ and $y$ have the same itinerary up to time $n$, the following inequality holds:
$$   e^{-c}<\frac{(G ^n)'(x) } {(G ^n)'(y) } < e^c$$ 
\em

\noindent The proof of the Bounded Distortion Lemma  can be found for instance in \cite[Thm 1, Chap 4, p.58]{PalisTakens}.
 
Substituting the inequality of the Bounded Distortion Lemma in Equality (\ref{equation333}) we deduce
	$$\frac{|I^*_{\underline{a}_n}|} {|I_{\underline{a}_n}|} > e^{-c} \cdot r_0. $$
	
	Then 
	\begin{equation}
		\label{equation301}
			\sum _{\underline{a}_n\in \{0,1\}^n } |I^*_{\underline{a}_n}|  >  e^{-c} \cdot r_0 \cdot \sum _{\underline{a}_n \in \{0,1\}^n } |I_{\underline{a}_n}|.
	\end{equation}

Denote 
$$A_n := \bigcup _{\underline{a}_n \in \{0,1\}^n } I_{\underline{a}_n}, \ \ U_n := \bigcup _{\underline{a}_n \in \{0,1\}^n } I^*_{\underline{a}_n}. $$
More precisely, $A_n$ is the union of all the atoms of generation $n$, and $U_n$ is the union of all the gaps of generation $n$, one gap inside each atom of the same generation.  From the construction of the Cantor set dynamically defined by $G$, we have
$$ A_{n+1}= A_n\setminus U_n,  \ \mbox{hence } m(A_{n+1}) = m(A_n) -m(U_n), $$
where $m$ denotes the Lebesgue measure.

The atoms of generation $n$ are pairwise disjoint, then the gaps of generation $n$ contained in them are also pairwise disjoint and we obtain the following result using inequality (\ref{equation301}):  $$m(U_n) =\sum _{\underline{a}_n\in \{0,1\}^n } |I^*_{\underline{a}_n}|  >  e^{-c} \cdot r_0 \cdot \sum _{\underline{a}_n \in \{0,1\}^n } |I_{\underline{a}_n}| =  e^{-c} \cdot r_0 \cdot  m(A_n).$$	 

Then
$$m(A_{n+1}) < m(A_n) (1- e^{-c} \cdot r_0) = \lambda \cdot m(A_n) \ \mbox{where } 0<\lambda := 1- e^{-c} \cdot r_0 <1.$$
Since the above inequality holds for all $n \geq 1$ we deduce
$$m(A_n) < \lambda^{n-1} m(A_1) \ \mbox{for all } n \geq 1.$$
Applying Equality (\ref{equationCantorSet}) and taking into account that $A_{n+1} \subset A_n$ for all $n \geq 1$ we conclude
$$m(K)= m\left(\bigcap _{n=1} ^{\infty} A_n\right) = \lim _{n \rightarrow + \infty} m(A_n) \leq \lim _{n \rightarrow + \infty} \lambda^{n-1} m(A_1) = 0,$$
completing the proof of Proposition \ref{PropositionLebK=0 if C1+Holder}.
\end{proof}

\subsection{First example on $S^1$.} \label{SubsectionFirstExampleOnCircle}

In this subsection we will construct a $C^1\setminus C^{1+ \alpha}$- expanding map on the circle $S^1$ that has a pseudophysical measure $\mu_K$, that is indeed physical, supported on a Cantor set $K\subset S^1$ with positive Lebesgue measure. The physical measure $\mu_K$ will be Bernoulli (hence, ergodic) and absolutely continuous with respect to the Lebesgue measure, but not equivalent to it (although $\mu_K$ is equivalent to the restriction of the Lebesgue measure to $K$). Applying Theorem \ref{MainTheorem}, $\mu_K$  will satisfy Pesin's Entropy Formula. 

The construction of this example is based on Subsections \ref{SubsectionCantorSetsDynDefined} and \ref{SubsectionBowen} and taken from the first part of Example 2.1 of  \cite{Araujo-Santos-2019}.  

Consider the construction in Subsection \ref{SubsectionBowen} of the $C^1\setminus C^{1+ \alpha}$ - expanding map $G: I_0 \cup I_1 \to [-1,1]$, as in Figure \ref{Figure2}, and the Cantor set dynamically defined by $G$ such that $m(K )>0$ (recall that $m$ denotes the Lebesgue measure).  Construct any order preserving, surjective and $C^1$- expanding map $G_c: \overline {I_c^*} \to [-1,1]  $ such that
$$\lim _{x \rightarrow -b_0^+} G'_c(x) =\lim _{x \rightarrow b_0^-} G'_c(x) =2.$$

Construct $f:[-1,1] \to [-1,1]$  a piecewise  $C^1\setminus C^{1+ \alpha}$-expanding, with three continuity pieces (index 3 of the expanding map in the circle), order preserving and surjective in each continuity piece (see Figure \ref{Figure1}), as follows:

$$ f(x)  :=
\begin{cases} 
    G_0 (x) = G|_{I_0} (x)  & \mbox{ if } x \in I_0 =[-1, -b_0]  \\
	  G_c (x)  & \mbox{ if } x \in I^*_c =( -b_0, b_0)  \\
	 G_1 (x) = G|_{I_1} (x)  & \mbox{ if } x \in I_1=[b_0,1] 
\end{cases} 
$$

As said at the beginning of this section $f$ induces a $C^1\setminus C^{1+ \alpha}$ expanding map on the circle $S^1$.

\subsubsection {The invariant probability measure.}
Define the probabiliy measure
$$\mu_K (B) := \frac{m(B \cap K)}{m(K)} \ \mbox{ for any Borel set } B \subset S^1.$$

So $\mu_K $ is supported on the Cantor set $K$ and besides $$   \mu_K \ll m $$
Nevertheless, $m$ is not absolutely continuous with respect to $ \mu_K$ because the Lebesgue measure of the gaps is positive but their $\mu_K$-measure is zero.

\begin{proposition} \label{proposition mu_K invariant}
The probability measure	$ \mu_K$ is $f$-invariant.
\end{proposition}

\begin{proof}
	Since the atoms of all generations intersected with the Cantor set $K$ generate the $\sigma$-algebra of $K$, to prove that $\mu_K$ is $f$-invariant it is enough to prove that
	$\mu_K(f^{-1} (A_{\underline{a}_n})) = \mu_K (A_{\underline{a}_n}) $ for all $ {\underline{a}_n} \in \{0,1\}^n$.
	By construction of the Cantor set $K $  dynamically defined by $G= f|_{I_0 \cup I_1}$, all the $2^n$ atoms of generation $n$ have the same Lebesgue measure. We conclude that $\mu_K (A_{\underline{a}_n}) = 1/2^n $. Also 
	$f^{-1} (A_{\underline{a}_n}) = G^{-1} (A_{a_0 a_1 \ldots a_{n-1}})  =  A_{0 a_0 a_1 \ldots a_{n-1}} \cup A_{1 a_0 a_1 \ldots a_{n-1}} $, which are two disjoint atoms of generation $n+1$. Since all the $2^{n+1}$ atoms of generation $n+1$  have the same Lebesgue measure, we obtain
	$$\mu_K(f^{-1} (A_{\underline{a}_n})) =  \mu_K(  A_{0 a_0 a_1 \ldots a_{n-1}}) +  \mu_K(  A_{1 a_0 a_1 \ldots a_{n-1}}) = \frac{ 1}{2^{n+1} } + \frac{ 1}{2^{n+1} }  = \frac{ 1}{2^{n} } = \mu_K (A_{\underline{a}_n} ),  $$ 
	as needed.
	\end{proof} 
	
	\subsubsection {The Bernoulli property.}
\begin{proposition} \label{proposition mu_K Bernoulli}
	The probability measure	$ \mu_K$ is  Bernoulli for $f$.  Precisely, it is equivalent, through a bimeasurable conjugation $h: K \mapsto 2^{\mathbb N} $, to the Bernoulli measure $\nu$ on the shift space $2^{\mathbb N}$ of all the sequences of 0's and 1's that assigns the same weight $1/2$ to each symbol 0 or 1.
\end{proposition}	
\begin{proof}
	Consider the shift space $2^{\mathbb N}$ of two symbols composed by all the sequences of 0's and 1's. Precisely,
	$$2^{\mathbb N} := \{\underline {a} = a_0 a_1 \ldots a_n \ldots: \ a_n \in \{0,1\} \ \mbox{for all } n\geq0 \}. $$
	Consider the shift map to the right:
	$ \sigma : 2^{\mathbb N} \to 2^{\mathbb N} $ defined by
	$$\sigma ( a_0 a_1 \ldots a_n \ldots  ) =  a_1 \ldots a_n \ldots $$
	and the $\sigma$-invariant Bernoulli measure  $\nu$ in $2^{\mathbb N}$ giving to each digit 0 or 1 the same weight $1/2$.
	
	Define $h: K \to 2^{\mathbb N} $, assigning to each point $x \in K$ its itinerary $a_0 a_1 \ldots a_n \ldots$, namely $f^n(x) \in I_{a_n}$ for all $n \geq 0$. As seen in Subsection \ref{SubsectionCantorSetsDynDefined} $h$ is invertible. Besides, from the construction of the atoms of $K$ when applying $G= f|_{I_0 \cup I_1} $ to an atom $I_{a_0 a_1 \ldots a_{n-1}}$ of generation $n\geq 2$, one obtains the atom $I_{ a_1 \ldots a_{n-1}} $ of generation $n-1$. So, 
	$$h \circ f(x) =     \sigma \circ h (x) \ \mbox { for all } x\in K, $$
 and $h$ is a measurable conjugation, with measurable inverse, between $f|_K$ and the shift $\sigma$.
	To prove that $\mu_K$ is Bernoulli, we will prove that
	\begin{equation}
		\label{equation309}
			\mu_K = h^* \ \nu \  \ ,
	\end{equation}
	where $h^*$ is the pull forward defined by $h^*\nu(B) = \nu (h(B)) $ for any Borel set $B \subset K $.
	
	Since the atoms of all generations, intersected with $K$, generate the Borel $\sigma$-algebra of $K$, it is enoug to prove Equality (\ref{equation309}) for $B= I_{\underline{a}_n}$ for any $\underline a _n \in \{0,1\}^n$ and any $n\geq 1$.
	
	On the one hand, as said at the beginning, the measure $\mu_K$ of any atom of generation $n$ is $1/2^n$.
	On the other hand, 
	$h(I_{\underline{a}_n}) $ is the cylinder  $ C_{\underline{a}_n} \subset 2^{\mathbb N}$, where ${\underline{a}_n}= a_0 a_1 \ldots a_{n-1}$ is fixed. Precisely, the cylinder $ C_{\underline{a}_n} $ is the set of all the sequences of 0's and 1's such that their first $n$ terms are fixed and equal to $ a_0 a_1 \ldots a_{n-1}$.
	
	So $h^* \nu (I_{\underline{a}_n})= \nu (h(I_{\underline{a}_n})) = \nu (C_{\underline{a}_n} ) = p(a_0) \cdot p(a_1) \cdot \ldots \cdot p(a_{n-1})$, where $p(a_i)$ is the weight of the symbol $a_i$ of the Bernoulli measure $\nu$. Since both symbols 0 and 1 are equally weighted by $1/2$ we conclude that $h^* \nu (I_{\underline{a}_n}) = 1/2^n = \mu_K(I_{\underline{a}_n})$, as stated.	
\end{proof}

\begin{proposition} \label{Proposition mu is Physical}
	The measure $\mu_K$ is physical.
\end{proposition}
\begin{proof}
Recall Definition  \ref{DefinitionStatisticalBasin} of the basin of statistical attraction of a probability measure. Since $\mu_K$ is Bernoulli, it is ergodic and the sequence of empiric probabilities $\sigma_n(x)$ converges to $\mu_K$ for $\mu_K$-almost every point $x$. In other words, the basin $B(\mu_K)$ of statistical attraction of $\mu_K$ contains $\mu_K$-almost all the points. Since $\mu_K$ is absolutely continuous with respect of the Lebesgue measure $m$, we conclude that $m(B(\mu_K))  >0$ and applying Definition \ref{DefinitionPhysicalMeasure}, $\mu_K$ is physical.
\end{proof}

%There may exist other pseudo-physical measures in this example. Nevertheless, it can be proved (the proof is rather difficult) that, if the map $f$ is adequately chosen in the gaps of $K$ then
% the basin $B(\mu_K)$ of statistical attraction of $\mu_K$ covers Lebesgue almost all the points of the interval $[-1,1]$. If so, applying Theorem 1.5 of \cite{CatEnr-polaca}, we conclude that $\mu_K$ is the unique pseudo-physical measure. 

\subsection{Second example on $S^1$ }  \label{SubsectionSecondExampleOnCircle}

In this subsection we will construct a $C^1\setminus C^{1+ \alpha}$- expanding map on the circle $S^1$ that has finitely many pseudo-physical measures $\mu_{K_i},  i = 1, 2, \ldots m$, that are indeed physical, supported on pair disjoint  Cantor sets $K_i \subset S^1$ with positive Lebesgue measure. All these physical measures $\mu_{K_i}$ will be Bernoulli (hence, ergodic) and absolutely continuous with respect to the Lebesgue measure. Since they are physical, applying Theorem \ref{MainTheorem}, all these measures $\mu_{K_i}$  will satisfy Pesin's Entropy Formula. 

The construction of this example is based on the example of Subsection \ref{SubsectionFirstExampleOnCircle} and  taken from  Example 2.1 of  \cite{Araujo-Santos-2019}.
The method consists of modifying the map $f$ of the Example of Subsection \ref{SubsectionFirstExampleOnCircle} inside a gap of the Cantor set $K$ dynamically defined by $f$, by adding one more continuity piece in this gap. After that, we will construct inside this gap a new Cantor set with positive Lebesgue measure by applying again  Bowen's construction explained in Subsection \ref{SubsectionBowen}.  The method can be repeated finitely many times, choosing at any time, a different gap of the Cantor sets previously constructed. Nevertheless, this method can not be repeated infinitely many times because each time it is applied the number of continuity pieces in the interval $[-1,1]$ (the index of the induced map on the circle $S^1$)
increases.

To fix ideas, we will explain with details the construction of a second Cantor set with positive Lebesgue measure and a new continuity piece inside the gap $I^*_c =(-b_0, b_0)$ of Figure \ref{Figure1}.

Let $f_1$ be the expanding map on the circle of the example in Subsection \ref{SubsectionFirstExampleOnCircle} and denote by $K_1$ its dynamically defined Cantor set contained in the intervals $I_0 \cup I_1$ as in Figure \ref{Figure1}, being $I^*_c$ its gap of generation 0.   Choose two real numbers $c_1$ and $b_1$ such that
$$0<b_1<c_1<b_0<1.$$
Substitute $f_1|_{I^*_c}$, being $I_c^*$ the gap of generation 0, by any pair of $C^1$-expanding order preserving and surjective maps $g_0: [-b_0, 0] \to [-1,1] $ and $g_1 : [0, b_0] \to [-1,1] $ as in Figure \ref{Figure4}, such that 
$$g_0(-c_1) = - c_1, \ \ g_0(-b_1) = c_1, \ \ g_1(b_1) = -c_1, \ \ g_1(c_1) = c_1, $$
$$ \lim_{x \rightarrow -b_0^+} g'_0 (x) = \lim_{x \rightarrow b_0^-} g'_1(x) = 2,     \    \lim_{x \rightarrow 0^-} g'_0 (x) = \lim_{x \rightarrow 0 ^+} g'_1(x).  $$

Construct $g: [-1,1] \mapsto [-1,1]$ with four order preserving, surjective and expanding pieces, as follows (see Figure \ref{Figure4}):
$$ g(x)  :=
\begin{cases} 
f_1|_{I_0}(x) =	G_0 (x)  & \mbox{ if } x \in I_0 =[-1, -b_0]  \\
	g_0(x) & \mbox{ if } x \in( -b_0, 0] \\
		g_1(x) & \mbox{ if } x \in( 0, b_0) \\
f_1|_{I_1}(x) =	G_1 (x) & \mbox{ if } x \in I_1=[b_0,1] 
\end{cases} 
$$
We have that $-c_1 \in (-b_0, 0)$ is the fixed point of $g_0$ and  $c_1 \in (0, b_0)$ is the fixed point of $g_1$.  Observe that, in the interval $[-c_1, c_1]$ there are two disjoint closed subintervals $J_0$ and $J_1$ separated by an  open interval (a gap) $J^*_c$, defined by:
$$ J_0 := [-c_1, -b_1] =g_0^{-1} ([-c_1, c_1]) , \ \ \ J_1:=[b_1, c_1]=g_1^{-1} ([-c_1, c_1]), $$  $$ J_c^*:= (-b_1, b_1) = [-c_1, c_1]\setminus (J_0 \cup J_1).$$
As seen in Subsection \ref{SubsectionCantorSetsDynDefined}, the $C^1$-expanding maps $g_0|_{J_0}: J_0 \to [-c_1,c_1] $ and $g_1|_{J_1}: J_1 \to [-c_1,c_1]$ dynamically define a Cantor set contained in $J_0 \cup J_1$.

Now, substitute $g_0|_{J_0}$ and $g_1|_{J_1}$ by   $C^1\setminus C^{1+ \alpha}$- expanding, order preserving and surjective maps  $$ H_0 : J_0 \mapsto [-c_1, c_1], \ \ \ H_1 : J_1 \mapsto [-c_1, c_1]$$
respectively, such that the Cantor set $K_2$ dynamically defined by $H_0$  and $H_1$ in $J_0 \cup J_1$ has positive Lebesgue measure. Such maps $H_0$ and $H_1$ and the Cantor set $K_2$ are constructed applying Bowen's method explained in Subsection \ref{SubsectionBowen}, after replacing the interval $[-1,1]$ by $[-c_1, c_1]$, the atoms of generation 1, $I_0$ and $I_1$, by $J_0$ and $J_1$ respectively, and the gap $I^*_c= (-b_0, b_0)$ by the gap $J_c^*= (-b_1, b_1)$.
 
 \begin{figure}[ht]
 	\centering
 	\hspace{-0.2in}\vspace{0.2in}\includegraphics[scale=.40]{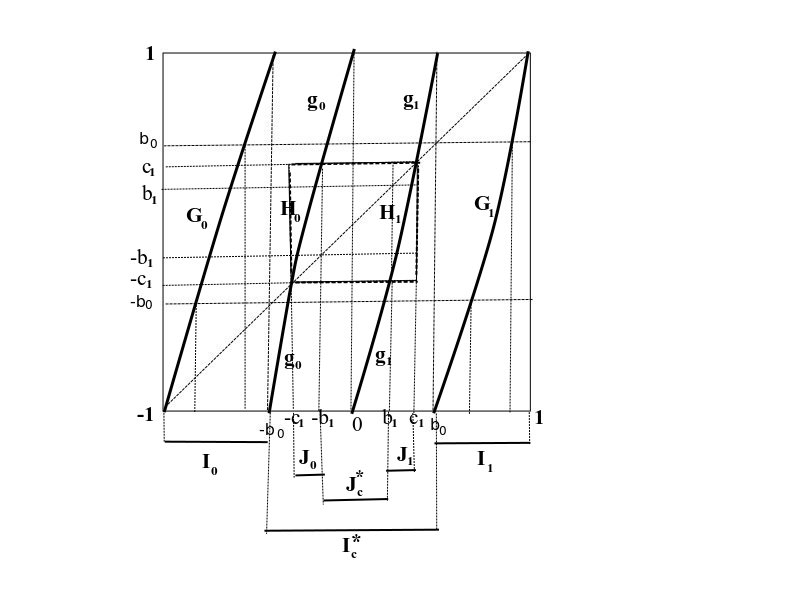}
 	\caption{The maps $G_0 = f|_{I_0}: I_0 \to [-1,1]$ and $G_1 = f|_{I_1}: I_1 \to [-1,1]$ dynamically define  the first Cantor set $K_1\subset I_0 \cup I_1$ with positive Lebesgue measure, whose atoms of generation 1 are $I_0$ and $I_1$. The maps   $H_0 = f|_{J_0}: J_0 \to [-c_1,c_1]$ and $H_1 = f|_{J_1}: J_1 \to [-c_1,c_1]$ dynamically define  the second Cantor set $K_2\subset J_0 \cup J_1$,  also with positive Lebesgue measure, whose atoms of generation 1 are $J_0$ and $J_1$. \label{Figure4}}
 \end{figure}

Finally, modify $g_0$ (we will still call it $g_0$) in a small left neighborhood of $-c_1$ and in small right  neighborhood of $-b_1$ so it $C^1$-glues well with $H_0$ at $-c_1$ and $-b_1$. Analogously modify $g_1$ (we will still call it $g_1$) in a small left neighborhood of $b_1$ and in small right neighborhood of $c_1$ so it $C^1$-glues well with $H_1$ at $b_1$ and $c_1$ (see Figure \ref{Figure4}).

Define
$$ f(x)  :=
\begin{cases} 
	f_1|_{I_0}(x) =	G_0 (x)  & \mbox{ if } x \in I_0 =[-1, -b_0]  \\
	f_1|_{I_1}(x) =	G_1 (x) & \mbox{ if } x \in I_1=[b_0,1]  \\
	H_0(x)  & \mbox{ if }   x \in J_0 =[-c_1, -b_1]     \\
	H_1(x)  & \mbox{ if }    x \in J_1 =[b_1, c_1]      \\
	g_0(x) & \mbox{ if } x \in( -b_0, -c_1) \cup (-b_1, 0] \\
	g_1(x) & \mbox{ if } x \in( 0, b_1) \cup (c_1, b_0)
\end{cases} 
$$
By construction $f$ dynamically defines two Cantor sets $K_1 \subset I_0 \cup I_1$ and $K_2 \subset J_0 \cup J_1$ with positive Lebesgue measures.

Construct the following two probability measures supported on $K_1$ and $K_2$, respectively: 
$$\mu_{K_1} (B) := \frac{m (B \cap K_1)}{m(K_1)}, \ \ 
\mu_{K_2} (B) := \frac{m (B \cap K_2)}{m(K_2)}, $$
for any Borel set  $B \subset [-1,1]$,
where $m$ denotes the Lebesgue measure.

By construction $ \mu_{K_1} $ and $ \mu_{K_2}$ are absolutely continuous with respect to the Lebesgue measure. As proved in Subsection \ref{SubsectionFirstExampleOnCircle}, $ \mu_{K_1} $ and $ \mu_{K_2} $ are $f$-invariant, Bernoulli and physical.  Besides, as proved in that subsection, the basins of statistical attractions 
$B( \mu_{K_1}  ) $ and $B( \mu_{K_2}  ) $ contain Lebegue almost all points of $K_1$ and $K_2$ respectively.

%According to the statistical behaviour of the orbits that wander forever inside the gaps, there may exist other pseudo-physical measures besides $ \mu_{K_1} $ and $ \mu_{K_2} $. Nevertheless, it can be proved (although the proof is rather difficult) that, if the map $f$ is adequately chosen in the gaps of $K_1$ and $K_2$, then  the union $B( \mu_{K_1}  ) \cup B( \mu_{K_2}  ) $ of the basins of statistical attractions of $\mu_{K_1}   $ and $ \mu_{K_2}   $ cover Lebesgue a.e the whole interval $[-1,1]$. So,  after applying Theorem   1.5  of \cite{CatEnr-polaca}, there does not exist other pseudo-physical measures different from $\mu_{K_1}   $ and $ \mu_{K_2}   $.
%
% 
%
%

\subsection{Example on $T^2$.} \label{SubsectionFirstExampleOnTorus}

In this subsection we will construct a $C^1\setminus C^{1+ \alpha}$- expanding map on the 2-torus $T^2$ that has  finitely many pseudophysical measures $\mu_{K_i},  i = 1, 2, \ldots k$, that are indeed physical, supported on pair disjoint  Cantor sets $K_i \subset T^2$ with positive Lebesgue measure. All these physical measures $\mu_{K_i}$ will be Bernoulli (hence, ergodic) and absolutely continuous with respect to the Lebesgue measure, but not equivalent to it.  Applying Theorem \ref{MainTheorem}, all these measures $\mu_{K_i}$  will satisfy Pesin's Entropy Formula.   

The construction is based on the examples of Subsections \ref{SubsectionFirstExampleOnCircle} and \ref{SubsectionSecondExampleOnCircle}.

Since $T^2 = S^1 \times S^1$ we will construct $f: T^2 \to T^2$ taking two maps $f_1, f_2: S^1 \to S^1$ and defining\begin{equation}
	\label{equation350}
	f(x.y) : = (f_1(x), f_2(y))  \mbox{ for all } (x,y) \in T^2
\end{equation}
Choose $f_1$ the expanding map in the circle constructed in the example of Subsection \ref{SubsectionFirstExampleOnCircle} and $f_2$ the expanding map in the circle constructed in the example of Subsection \ref{SubsectionSecondExampleOnCircle}.

\begin{proposition} \label{Proposition 5-8}
	The map $f$ defined by Equality (\ref{equation350}) is $C^1\setminus C^{1 + \alpha}$- expanding on the torus $T^2$.
\end{proposition}
\begin{proof}
	First, $f$ is $C^1\setminus C^{1 + \alpha}$ because $f_1$ and $f_2$ also are.
	Second, let us prove that $f$ is expanding.
	Recall Definition \ref{DefinitionExpandingMaps} of expanding  maps on manifolds. Take $(x,y)\in T^2$ and $v =(v_1, v_2) $ in the tangent space $ T _{(x,y)} T^2$.
	
	We have 
	$$\|Df_{(x,y)} v\| = \|Df_{(x,y)} (v_1, v_2)\| = \| f'_1(x) v_1, f'_2(y) v_2\| =$$ $$ \sqrt {(f'_1(x) v_1) ^2 + (f'_2(x) v_2) ^2}
	\geq \gamma \sqrt{v_1^2 + v_2^2} = \gamma \|v||,$$
	where $\gamma= \min \{\min_{x \in S^1} f'_1(x), \min_{y \in S^1} f'_2(y) \} $. As $\gamma >1$ because $f_1$ and $f_2$ are expanding maps on the circle, we conclude that $f$ is expanding on the torus, as stated.
\end{proof}

Denote by $m_{T^2}$ the Lebesgue measure on the torus, and $m_{S^1}$ the Lebesgue measure on the circle.
We have
\begin{equation}
	m_{T^2} = m_{S^1} \times  m_{S^1}, \ \ \ m_{T^2}(B) = m_{S^1}(B_1) \cdot  m_{S^1}(B_2)
	\label{equation881}
\end{equation}
for any measurable rectangle $B= B_1 \times B_2$ on the torus, where $B_1$ and $B_2$ are Borel sets in the circle.

Analogously, for two given finite measures $\mu_{ M_1}$ and $\mu_{ M_2}$ on the measurable spaces $M_1$ and $M_2$ respectively, the product measure  $\mu := \mu_{M_1}\times \mu_{M_2}$ on the product space $M_1 \times M_2$, satisfies

\begin{equation}
	\label{equation80b}
\mu (B) =\mu_{M_1}(B_1) \cdot  \mu_{M_2}(B_2)
\end{equation}
for any rectangle   $B= B_1 \times B_2$, where $B_1 \subset M_1$ and $B_2 \subset M_2$ are measurable sets.

On the one hand, from the construction of Subsection  \ref{SubsectionFirstExampleOnCircle}, the map $f_1$ on $S^1$  dynamically defines a Cantor set $K_{S^1}$ with positive Lebesgue measure, that is the support of an $f_1$-invariant measure $$\mu_{K_{S^1}} \ll m_{S^1}.$$ This measure is Bernoulli and physical for $f_1$.
	
On the other hand, from the construction of Subsection \ref{SubsectionSecondExampleOnCircle}, the map $f_2$ on $S^1$ dynamically defines $k \geq 1$ pairwise disjoint Cantor sets $K_{1, S^1},  K_{2, S^1}, \ldots,  K_{k, S^1}$, all with positive Lebesgue measure, that are the supports of $k$ different $f_2$-invariant measures
$$\mu_{K_{1, S^1}}, \mu_{K_{2, S^1}}, \ldots, \mu_{K_{k, S^1}} \ll m_{S^1}. $$ These measures are all Bernoulli and physical for $f_2$.

Construct on the torus the following $k$ pairwise disjoint Cantor sets:
$$K_{i, T^2}:= K_{S^1} \times K_{i, S^1} \ \mbox{for all } 1\leq i \leq k.$$
They all have positive Lebesgue measure on the torus. In fact, applying Equality (\ref{equation881}) we have
$$ m_{T^2} (K_{i, T^2}) = m_{S^1} (K_{S^1}) \cdot m_{S^1} (K_{i, S^1}) >0 \ \mbox{for all } 1\leq i \leq k.$$

From the constructions of Subsections \ref{SubsectionFirstExampleOnCircle} and \ref{SubsectionSecondExampleOnCircle} the following probabilities measures on $S^1$ are $f_1$ and $f_2$ invariant, respectively, Bernoulli and physical : 
$$\mu _{K_{S^1}} (B_1) :=  \frac{ m_{S^1} (K_{S^1} \cap B_1)}{ m_{S^1} (K_{S^1})} \ \ \forall \ \mbox{Borel set } B_1 \subset S^1, $$
$$\mu _{K_{i, S^1}} (B_2) :=  \frac{ m_{S^1} (K_{i, S^1} \cap B_2)}{ m_{S^1} (K_{i, S^1})} \ \ \forall \ \mbox{Borel set } B_2 \subset S^1 , \ \ \forall \ i\in \{1, 2, \ldots, k\}. $$

Construct on $T^2$ the following probability measures
$$\mu _{K_{i, T^2}} := \mu _{K_{S^1}} \times \mu _{K_{i, S^1}} \ \ \forall \ i\in\{1,2,\ldots, k\}. $$
It is immediate to check that $\mu _{K_{i, T^2}}$ is supported on the Cantor set $K_{i, T^2}$ and is $f$-invariant. In fact, it is enough to check that $\mu _{K_{i, T^2}}  (f^{-1} (B)) = \mu _{K_{i, T^2}} (B)$ for any measurable rectangle $B= B_1 \times B_2$. 

\begin{proposition} \label{Proposition 5-9} 
	For all $i \in \{1,2, \ldots, k\} $ and for all Borel set $B \subset T^2$ the following equality holds:
	$$ \mu _{K_{i, T^2}}(B) = \frac{ m_{T^2} (K_{i, T^2} \cap B)}{ m_{T^2} (K_{i, T^2})} .    $$
	Hence, 
		$$ \mu _{K_{i, T^2}} \ll  m_{T^2}.$$
	
\end{proposition}
\begin{proof}
	Since the measurable rectangles generate the Borel $\sigma$-algebra of $T^2$ it is enough to check the equality of the proposition for $B= B_1 \times B_2$, where $B_1, B_2 \subset S^1$ are Borel sets in the circle.
	In fact, applying Equalities \eqref{equation881} and  \eqref{equation80b} , we obtain:
	$$ \frac{ m_{T^2} (K_{i, T^2} \cap (B_1 \times B_2))}{ m_{T^2} (K_{i, T^2})} =
	 \frac{ m_{T^2} ((K_{S^1} \times K_{i, S^1} ) \cap (B_1 \times B_2))}{ m_{T^2} (K_{S^1} \times K_{i, S^1} )} =
		$$  
		$$\frac{ m_{T^2} ((K_{S^1} \cap B_1) \times (K_{i, S^1} \cap B_2))}{ m_{S^1} (K_{S^1}) \cdot   m_{S^1} (  K_{i, S^1} )}  =     \frac{ m_{S^1} (K_{ S^1} \cap B_1) \cdot m_{S^1} (K_{i, S^1}\cap B_2)}{ m_{S^1} (K_{S^1}) \cdot   m_{S^1} (  K_{i, S^1} )} =$$ 
	$$	\mu _{K_{S^1}} (B_1) \cdot \mu _{K_{i, S^1}} (B_2)  =   \mu _{K_{i, T^2}}(B_1 \times B_2).  $$
\end{proof}
\begin{proposition}

	For all $i\in \{1,2,\ldots\}$ the probability measure	$  \mu _{K_{i, T^2}}$ is  Bernoulli.  Precisely, it is equivalent to the Bernoulli measure $\nu$ on the shift space $4^{\mathbb N}$ of all the sequences of symbols $0,1,2$ or $3$ that assigns the same weight $1/4$ to each symbol.
\end{proposition}

\begin{proof}
	From Proposition \ref{proposition mu_K Bernoulli} the measures 
$	\mu _{K_{S^1}}    $ and $\mu _{K_{i, S^1}}  $  on the circle are	
equivalent, through  bimeasurable conjugations $h_1: K_{ S^1}\mapsto 2^{\mathbb N} $ and $h_{2,i}: K_{i, S^1} \mapsto 2^{\mathbb N} $, to the Bernoulli measure $\nu_{(1/2,1/2)}$ on the shift space $2^{\mathbb N}$ of all the sequences of 0's and 1's that assigns the same weight $1/2$ to each symbol 0 or 1. Namely,
$$h_1 \circ f_1 (x) =     \sigma_{2^\mathbb N} \circ h_1 (x) \ \mbox { for all } x\in K_{S^1} \subset S^1, $$
$$h_{2,i}\circ f_2 (y) =     \sigma_{2^\mathbb N} \circ h_{2,i} (y) \ \mbox { for all } y\in K_{i, S^1} \subset S^1, $$
where $\sigma_{2^\mathbb N}$ is the shift to the right in the space $2^{\mathbb N}$, and from Equality \eqref{equation309}:
$$ \mu_{K_{S^1}} = h_1^* \ \nu_{(1/2,1/2)}, \ \ \   \mu_{K_{i, S^1}} = h_{2,i}^* \ \nu_{(1/2,1/2)}. $$
So
\begin{equation}
	\label{equation988b}
	\mu _{K_{i, T^2}} := \mu _{K_{S^1}} \times \mu _{K_{i, S^1}} = h_1^* \ \nu_{(1/2,1/2)} \times   h_{2,i}^* \ \nu_{(1/2,1/2)}  
\end{equation}

For any point $x \in {K_{S^1}}$ consider the sequence $h_1(x)= a_0, a_1, \ldots, a_n,\ldots \in 2^{\mathbb N}$ and denote   $(h_1(x))_n:=a_n$ for all $n \geq 0$. Analogously, for any point $y \in {K_{i, S^1}}$ consider the sequence $h_{2,i}(y)= b_0, b_1, \ldots, b_n,\ldots \in 2^{\mathbb N}$ and denote   $(h_{2,i}(y))_n := b_n$ for all $n \geq 0$.

Define $h: K_{i, T^2} \mapsto 4^{\mathbb N}$ as follows:
$$h(x,y) = c_0, c_1, \ldots, c_n, \ldots   \mbox {such that } c_n \in\{0,1,2,3\}  \mbox{ is }$$

the symbol 0 if $(h_1(x))_n \ (h_{2,i}(y))_n = 00    $, 

the symbol 1 if $(h_1(x))_n \ (h_{2,i}(y))_n = 01    $, 

the symbol 2 if $(h_1(x))_n \ (h_{2,i}(y))_n =10    $,  

the symbol 3 if $(h_1(x))_n \ (h_{2,i}(y))_n = 11    $.

We denote such a correspondence as $(h(x,y))_n = (h_1(x))_n \ (h_{2,i}(y))_n$

So, we have the following equalities for  any point $(x,y)  \in K_{i, T^2} $
$$(h \circ f (x,y))_n = (h (f_1(x), f_2(y)))_n = (h_1(f_1(x))_n \ (h_{2,i}(f(y))_n  = $$
$$ (\sigma_{2^{\mathbb N}}(h_1 (x)))_n \    (\sigma_{2^{\mathbb N}}(h_{2,i} (y)))_n =
	\sigma_{4^{\mathbb N}}((h_1(x))_n \  (h_{2,i} (y))_n     ) = \sigma_{4^{\mathbb N}} ((h(x,y))_n ),
	$$
where $ 	\sigma_{4^{\mathbb N} } $ is the shift to the right in the space $4^{\mathbb N}$. We have proved that
$$ h \circ f = 	\sigma_{4^{\mathbb N} } \circ h,$$
or, in other words, $f$ and the shift are conjugated by $h$. Besides $h$ is invertible because $h_1$ and $h_{2,i}$ also are.

Now, to end the proof we need to check that
\begin{equation}
	\label{equationToBeProved988}
	\mu_{K_{i, T^2}} = h^* \nu, 
\end{equation} 
where $\nu$ is the Bernoulli measure on the shift space $4^{\mathbb N}$ that assigns to each symbol the weight $1/4$.

Since the rectangles in $T^2= S^1 \times S^1$ generate the Borel  $\sigma$-algebra of $T^2$, let us check the above equality for any $B = B_1 \times B_2 \subset T^2$, where $B_1$ and $B_2$ are Borel sets in $S^1$.
In fact
$$ h^* \nu  (B_1 \times B_2) = \nu ( h(B_1 \times B_2)) = \nu (h_1(B_1)  \times h_{2,i} (B_2)) $$
In the shift spaces, the  Bernoulli measure $\nu $ in $4 ^{\mathbb N}=2 ^{\mathbb N} \times 2 ^{\mathbb N} $, which assigns to each of the four symbols the same weight $1/4$, is the product measure $\nu_{(1/2,1/2)} \times \nu_{(1/2,1/2)} $. 

Therefore, taking into account that $h_1(B_1)  \times h_{2,i} (B_2)$ is a rectangle in $4^{\mathbb N}$, and applying Equality \eqref{equation988b}, we obtain 
$$h^* \nu  (B_1 \times B_2) = \nu (h_1(B_1)  \times h_{2,i} (B_2))  =  \nu_{(1/2,1/2)}(h_1(B_1) )\ \cdot \  \nu_{(1/2,1/2)}( h_{2,i} (B_2)) =$$ $$ 
h_1^* \nu_{(1/2,1/2)}(B_1) \cdot   h_{2,i}^* \nu_{(1/2,1/2)} (B_2) =$$ $$ (h_1^* \nu_{(1/2,1/2)}\times  h_{2,i}^* \nu_{(1/2,1/2)}) (B_1 \times B_2)  =  \mu_{K_{i, T^2}} (B_1 \times B_2), $$
ending the proof of Equality \eqref{equationToBeProved988}).
\end{proof}
\begin{proposition} 
	For all $i \in \{1,2,\ldots, k\}$, the probability measure  $\mu_{K_{i, T^2}}$ is physical.
\end{proposition}
\begin{proof}
	Repeat the proof of Proposition \ref{Proposition mu is Physical} using the measure $\mu_{K_{i, T^2}}$ on the torus  instead of  the measure $\mu_K$ on the circle.	
\end{proof}

%%%%%%%%%%%%%%%%%%%%%%%%%%%%%

%%%%% %%%%%%%%%%%%%%%%%%%%%%%%

\end{document}